\documentclass[preprint,3p,12pt]{elsarticle}
\usepackage{mathrsfs}
\usepackage{amsmath}
\usepackage{amsthm}
\usepackage{stmaryrd}
\usepackage{bbding}
\usepackage{dcolumn}
\usepackage{graphicx}
\usepackage{amsfonts}
\usepackage{amssymb}
\usepackage{psfrag}
\usepackage{wrapfig}
\usepackage{subfigure}
\usepackage{makeidx}
\usepackage{bm}
\usepackage{epsf}
\usepackage{epsfig}
\usepackage{setspace}
\usepackage{graphicx}
\usepackage{epstopdf}
\usepackage{psfrag}
\usepackage{subfigure}
\usepackage{color}
\usepackage{mathrsfs}
\usepackage{amsbsy}
\usepackage{subeqnarray}
\usepackage{cases}
\usepackage{caption}
\usepackage{multirow}

\begin{document}
\title{Superconvergence of the Direct Discontinuous Galerkin Method for Two-Dimensional Nonlinear Convection-Diffusion Equations}

\author[BNU]{Xinyue Zhang}
\ead{201921130059@mail.bnu.edu.cn}

\author[BNU]{Waixiang Cao\corref{cor}}
\ead{caowx@bnu.edu.cn}
\cortext[cor]{Corresponding author, The research  was supported in part by National Natural Science Foundation of China grant 11871106.}
 
% \author[iapcm]{Wenjun Sun}
% \ead{sun\_wenjun@iapcm.ac.cn} 
\address[BNU]{School of Mathematical Sciences, Beijing Normal University, Beijing, 100875, China}
% \address[iapcm]{Institute of Applied Physics and Computational Mathematics, Beijing, China}
% \address[HKUST1]{Department of Mathematics, Hong Kong University of Science and Technology, Clear Water Bay, Kowloon, Hong Kong} 
%\address[HKUST2]{Shenzhen Research Institute, Hong Kong University of Science and Technology, Shenzhen, China}

\newtheorem{theorem}{Theorem}[section] %
\newtheorem{lemma}{Lemma}[section] %
\renewcommand{\proofname}{\rm \textbf{Proof}}

\newtheorem{corollary}{Corollary}[section] %
\newtheorem{remark}{Remark}[section]
\newtheorem{example}{Example}[section]

% \listoffigures

\begin{abstract}
	This paper is concerned with  superconvergence properties of the direct discontinuous Galerkin (DDG) method for two-dimensional nonlinear convection-diffusion equations.
	 By using the idea of correction function, we prove that,  for  any  piecewise tensor-product polynomials of degree  $k\geq 2$, 
	 the DDG solution is superconvergent at nodes and Lobatto points, with an order of ${\cal O}(h^{2k})$ and ${\cal O}(h^{k+2})$, respectively. 
        Moreover, superconvergence properties  for the derivative approximation are also studied and the superconvergence points are identified 
        at Gauss points,  with an order of ${\cal O}(h^{k+1})$. 
	 Numerical experiments are presented to  confirm the sharpness of  all the theoretical findings.
 \end{abstract}


\maketitle

\section{Introduction} 

 In this paper, we investigate the superconvergence behavior of the  DDG method for 
 the following two-dimensional nonlinear convention-diffusion equations
 \begin{equation}\label{2dCD}
	\begin{cases} 
	  u_t+ \nabla\cdot {\bf f}(u)=\triangle u, &(x,y,t)\in [0,2\pi] \times [0,2\pi] \times (0,T],\\
	  u(x,y,0)=u_0(x,y),& (x,y)\in [0,2\pi]\times [0,2\pi], 
	\end{cases} 
	\end{equation}
where  $u_0$, ${\bf f}(u)=(f_1(u),f_2(u))$  are  smooth functions.   For simplicity, we only consider the periodic boundary condition. 

The discontinuous Galerkin (DG) method is a class of finite element method using a completely discontinuous polynomial space 
for the numerical solution and the test functions.  Due to its   flexibilities such as the allowance of hanging nodes, $p$-adaptivity,  extremely local data structure, 
high parallel efficiency,  the DG method has been successfully applied to solve first order PDEs such as hyperbolic conservation laws (see, e.g., 
 \cite{DG-3,DG-2,DG-1,DG-5}).  However, for equations containing higher order spatial derivatives,  such as the convection-diffusion equation and the KdV equation, 
 the DG method cannot be directly applied due to
 the discontinuous solution space at the element interfaces, which is not regular enough to handle higher derivatives. 
 Several DG methods have been suggested in the literature to solve this problem,  including the local discontinuous Galerkin(LDG) method  in \cite{LDG-3,LDG-1,LDG-5},
 % originally  by Cockburn and Shu and further studied in \cite{LDG-2,LDG-3,LDG-4,LDG-5}, 
  the compact DG (CDG) method \cite{CDG-1};  the interior penalty (IP) methods,
 \cite{IPDG-1,IPDG-2,IPDG-4}; the direct discontinuous Galerkin (DDG) method \cite{Liu-2010,OtherDG-1,OtherDG-2}, as well as other DG schemes designed from a practical perspective.

The DDG method was first introduced by Liu and Yan in 2009  for convection-diffusion problems \cite{Liu-2009}. 
The basic idea of this method is to directly force the weak solution formulation of the PDE into the DG function space for both the numerical solution and test functions, without introducing any auxiliary variables or rewriting the original equation into a larger first-order system.  One main advantage of the DDG method lies in that: 
 it is compact in the sense that only the degrees of freedom belonging to neighboring elements are connected in the discretization, which means lower storage requirements and higher computational performance than noncompact schemes.  
 Compared with  LDG methods, DDG methods  need to choose special numerical fluxes for the solution derivative at cell interfaces  to ensure its stability, 
which gives DDG methods extra flexibility and advantage.  It was proved in  \cite{maximum-1} that   the DDG method  satisfies strict maximum principle with at least third order of accuracy, while only second order can be obtained for IPDG and LDG methods \cite{maximum-2}.   For  lower order piecewise constant or linear approximations, 
the DDG method may  degenerate to the IPDG method and shares the same  desired properties of the IPDG solution. While for higher order $P_k (k \geq 2)$ approximation, 
DDG method have quite a few advantages over the IPDG method in some aspects,  e.g.,  superconvergence approximations (see \cite{Miao}). 
  
During  past decades,   superconvergence analysis has attracted many researchers’ attention and significant progress has been. 
 For an incomplete list of references, we refer to
   \cite{Babuska1996,Bramble.Schatz.math.com,ewing-lazarov-wang,Neittaanmaki1987,Schatz.Wahlbin1996}
   for finite element methods (FEM),
   \cite{Cai.Z1991,Cao;Zhang;Zou2012,Cao;zhang;zou:2kFVM} for finite volume methods (FVM), and
\cite{Adjerid;Massey2006,super-4,Cao-Shu-Yang-Zhang:nonlinear,Chen;Shu:SIAM2010,Yang;Shu:SIAM2012}
 for  DG methods, and \cite{super-cao-cd,Miao} for DDG methods. 
In \cite{super-cao-cd}, the authors studied superconvergence phenomenon of DDG methods for the linear convection-diffusion equation in the one-dimensional setting. 
 Under  a special choice of the numerical flux (especially the choice of the parameter in the DDG numerical flux),  a 
$(k+2)$-th order  superconvergence of the numerical solution at Lobatto points, $(2k)$-th order at nodes,  and a $(k + 1)$-th order of the derivative approximation 
at Gauss points, are obtained. 

The main purpose of the current work is to study and reveal the superconvergence phenomenon of the DDG method for two-dimensional nonlinear convection diffusion equation. To our best knowledge,
no superconvergence results of DDG methods applied to nonlinear convection-diffusion equations are available in the literature.  The main difficulties in the superconvergence analysis for nonlinear problems are  how to deal with the nonlinear
terms in the error equation and how to improve the error accuracy.  
To deal with the two issues, we  first  use Taylor expansion to linearize the error equation and then reduce
 the superconvergence analysis for nonlinear problems  into two parts: 
  a  lower order linear term and high order nonlinear term.   Then we use the idea of correction function
to deal with the linear part to obtain higher-order accuracy and finally achieve our superconvergence goal.  
To be more precise,   we establish superconvergence result for the DDG approximation at nodes, Lobatto points and Gauss points (derivative approximation), whose convergence 
rates are $2k$, $k+2$, and $k+1$, respectively. As we may recall, 
all these superconvergence results are similar to these for the one-dimensional linear problems in \cite{super-cao-cd}.   
It should be emphasized that the convergence analysis for nonlinear problems are much more  sophisticated than that for the linear case.

The rest of the paper is organized as follows. In section 2, we present the DDG schemes
and then discuss  the choice of numerical fluxes and stability results  for the equation \eqref{2dCD}.  Section 3 and 4 are the main part of this paper, where we construct a series of special correction functions, by identify out the lower order terms in the scheme. With suitable initial discretizations and careful choice of numberical fluxes, superconvergence results for Lobatto, Gauss points and nodes are proved respectively. In Section 5, we present several numerical examples to  validate our theoretical results. Some concluding remarks are provided in Section 6. 

Throughout this paper, we adopt standard notations for Sobolev spaces such as $W^{m,p}(D)$ on subdomain $D\subset\Omega$ equipped with the norm $\|\cdot\|_{m,p,D}$ and semi-norm $| \cdot|_{m,p,D}.$ When $D = \Omega$, we omit the index $D$. If $p=2$, we set $W^{m,p}(D) = H^m(D),~\|\cdot\|_{m,p,D} = \|\cdot\|_{m,D},$ and $|\cdot|_{m,p,D} = |\cdot|_{m,D}.$ We use the  notation $A\lesssim B$ to indicate that A can be bounded by B multilpied by a constant independent of the mesh size h.

\section{DDG schemes}

 Let $0 = x_{\frac 1 2} < x_{\frac 3 2}< \cdots <x_{N_x+\frac 1 2}=2\pi$ and $0=y_{\frac 1 2} < y_{\frac 3 2}< \cdots <y_{N_y+\frac 1 2}=2\pi$. 
 For any positive integer $r$,
  we define $\mathbb{Z} _r = \{1,2,\cdots ,r\}$, and denote  by ${\cal T}_h$ 
  the rectangular partition of  $\Omega:=[0, 2\pi]\times [0, 2\pi]$. That is
\begin{equation*}
	\mathcal{T} _h=\{\tau_{i,j}=[x_{i-\frac 1 2},x_{i+\frac 1 2}] \times [y_{j-\frac 1 2},y_{j+\frac 1 2}]:(i,j) \in \mathbb{Z} _{N_x}\times\mathbb{Z} _{N_y} \}.
\end{equation*}
Denote by $\tau_i^x=[x_{i-\frac 12}, x_{i+\frac 12}], \tau_j^y=[y_{j-\frac 12}, x_{j+\frac 12}]$, and 
$h= \max_{\tau\in{\cal T}_h} h_\tau$ the meshsize of  ${\cal T}_h$, where $h_{\tau}$  is the diameter of the element  $\tau\in {\cal T}_h$. 
We assume that the partition ${\cal T}_h$ is regular. 

 Define the discontinuous ﬁnite element space 
\begin{equation*}
	V^k_h:=\{v\in L^2(\Omega):v|_{\tau} \in \mathbb{Q} _k(x,y)=\mathbb{P} _k(x)\times \mathbb{P} _k(y),\ \ \forall \tau\in{\cal T}_h\},
\end{equation*}
where $\mathbb{P} _k$ denotes the space of polynomials of degree at most $k$ with coefficients as functions of $t$.  
Let $\tau_1$ and $\tau_2$ be two neighboring cells with a common edge $e$, 
we denote by,   for any function $w$,  $\{w\}$ and $[w]$
 the average and the jump  of $w$, respectively. 
That is, 
\[
   \{w\}|_e=\frac{1}{2}(w_1+w_2), \ \ [w]|_e=w_2-w_1. 
\]
 Here $w_i=w|_{\partial\tau_i}, i=1,2$, and 
 the jump term is  
calculated as a forward difference along the normal direction, which is defined to be oriented from $\tau_1$ to $\tau_2$.

%Since functions in $V_h^k$ may be discontinuous across element boundaries, 
%we denote $v(x_{i+\frac 1 2}^-,\cdot,\cdot), v(x_{i+\frac 1 2}^+,\cdot,\cdot)$ as the left and right 
%limits of v at $(x_{i+\frac 1 2},\cdot,\cdot)$, and adopt the following notations:
%\begin{equation*}
%\begin{split} 
%	[v]|_{x_{i+\frac 1 2}}:=v(x^+_{i+\frac 1 2},y,t) - v(x^-_{i+\frac 1 2},y,t), \quad &[v]|_{y_{j+\frac 1 2}}:=v(x,y^+_{j+\frac 1 2},t) - v(x,y^-_{j+\frac 1 2},t),\\
%	\{v\}|_{x_{i+\frac 1 2}}=\frac 1 2(v(x^+_{i+\frac 1 2},y,t) +v(x^-_{i+\frac 1 2},y,t)),  \quad &\{v\}|_{y_{j+\frac 1 2}}=\frac 1 2(v(x,y^+_{j+\frac 1 2},t) +v(x,y^-_{j+\frac 1 2},t)).
% \end{split} 
% \end{equation*}
The DDG method for \eqref{2dCD} is to find a $u_h \in V_h^k$ such that for any $v_h\in V^k_h$, 
\begin{equation}\label{DDGloc}
   (\partial_tu_h,  v_h)_{\tau}=({\bf f}-\nabla u, \nabla v_h)_{\tau}-\int_{\partial\tau} v_h ( \widehat{{\bf f}(u_h)}-\widehat {\nabla u_h}) \cdot {\bf n} \mathrm ds-B,\ \ \forall \tau\in{\cal T}_h, 
\end{equation}
where $(u,v)_\tau=\int_{\tau} uv dxdy$, 
 ${\bf n}$ denotes the unit outward  vector of $\tau$,  $\widehat{{\bf f}(u_h)}, \widehat {\nabla u_h}$ are numerical fluxes, 
and  $B$ is the interface correction defined as 
\[
   B:=\frac 1 2 \int _{\partial\tau} [u_h] \nabla v_h \cdot {\bf n} \mathrm ds. 
\]
%\begin{equation*}
%	\begin{split} 
%		B&=\frac 1 2
%	\{\int _{\tau_j^y}([u_h]v_h)|_{\partial \tau _i^x} dy+\int _{\tau _i^x}([u_h]v_h)|_{\partial \tau_j^y} dx\},
%	\end{split} 
%\end{equation*}
% Here  we adopt some notations, namely
% -(\widehat{f_1(u_h)} v_h|_{\partial \tau _i^x})_j-(\widehat{f_2(u_h)} v_h|_{\partial \tau_j^y})_i
%\begin{align*}
%	&(\varphi ,v)_{i,j} = \iint _{\tau_{i,j}} \varphi v \mathrm dx \mathrm dy,\quad (\varphi ,v)_{i} = \int _{\tau _i^x} \varphi v \mathrm dx, \quad \quad (\varphi ,v)_{j} = \int _{\tau_j^y} \varphi v \mathrm dy\\
%	&v|_{\partial \tau _i^x} = v(x^-_{i+\frac 1 2},y,t)-v(x^+_{i-\frac 1 2},y,t),\quad v|_{\partial \tau_j^y} = v(x,y^-_{j+\frac 1 2},t)-v(x,y^+_{j-\frac 1 2},t).
% \end{align*}

  To   ensure the stability as well as for the accuracy of the DDG method, the choice of   numerical fluxes is of great importance. 
  For the convection flux, we follow 
 \cite{Liu-2015} and take the entropy flux (or E-flux) as our numerical flux. That is,  at any $e\in \tau_1\cap\tau_2$,  let 
  $u_i=u|_{\partial\tau_i}, i=1,2$, and the normal direction  be oriented from $\tau_1$ to $\tau_2$, 
then $\widehat{{\bf f} }(u_1,u_2)= (\widehat{ f_1 }(u_1,u_2),\widehat{ f_2 }(u_1,u_2))$ satisfies the following condition
 \begin{equation}\label{con-flux}
     {\rm sign }(u_2-u_1) (\widehat{f_i }(u_1,u_2)-f_i(u))\le 0,\quad i=1,2,
 %		\int _e \displaystyle\frac{\widehat{f_n}(u_1,u_2)-f(\xi)\cdot\overrightarrow{n} } {[u]}[a]^2 ds \leq 0,
 \end{equation}
for all $u $ between $u_1$ and $u_2$.  For the diffusion term, we take the follows numerical fluxes: 
\begin{equation}\label{dif-flux}
	\widehat {\nabla u}:=(\widehat{u_x},\widehat{u_y}),\ 
		\widehat{u_x}:= \beta_0h^{-1}[u]+\{u_x\}+\beta_1h[u_{xx}],\quad \widehat{u_y}: = \beta_0h^{-1}[u]+\{u_y\}+\beta_1h[u_{yy}], 
\end{equation}
where $(\beta_0,\beta_1)$ satisfy the following stability condition
\begin{equation*}
	\beta_0 \geq \Gamma (\beta_1),
\end{equation*}
with 
\begin{equation*}
	 \Gamma (\beta_1)= \displaystyle {\sup_{v\in \mathbb{P} ^{k-1}[-1,1]}} \displaystyle\frac{2(v(1)-2\beta_1\partial_{\xi}v(1))^2}{\int_{-1}^1 v^2(\xi) \mathrm d\xi}.
\end{equation*}
Summing up all elements $\tau$ in  \eqref{DDGloc}  and using the periodic boundary condition, we have  
\begin{equation}\label{DDGglo}
	(\partial_t u_h ,v_h)+A(u_h,v_h)=F(u_h,v_h),\ \  \forall v_h\in V_h^k, 
\end{equation}
where $(u,v)=\sum_{\tau\in{\cal T}_h} (u,v)_{\tau}$,  and 
\begin{eqnarray}\label{DDGsubterm}
 && A(u,v)=(\nabla  u,\nabla  v) +\sum_{e\in{\cal E}_h} \int_e 
 ( [v]\widehat{\nabla u}+[u]\{\nabla v\}) \cdot {\bf n} \mathrm ds,\\  
	&& F(u,v)=( {\bf f}(u),\nabla v)+\sum_{e\in{\cal E}_h} \int_e [v]\widehat{{\bf f} (u)}\cdot {\bf n} \mathrm ds. 
\end{eqnarray}
  Here ${\cal E}_h$ denotes the set of all edges of ${\cal T}_h$. 
  
  Due to the special choice of numerical fluxes in \eqref{con-flux}-\eqref{dif-flux},  it was proved in \cite{Liu-2015} that the DDG scheme \eqref{DDGglo}
 is stable and   has Optimal error estimate in $L^2$ norm.  Moreover, there holds 
for any function $v \in V_h^k$  that 
\begin{equation}\label{coercivity}
	F(v,v)\leq 0,\quad A(v,v) \geq \gamma \|v\|_E^2,
\end{equation}
where $\gamma\in(0,1)$ is some positive constant, and 
\begin{equation}\label{energynorm}
	\|v\|_E^2=(\nabla v,\nabla v)+ \frac{\beta_0 }{h}\sum_{e\in{\cal E}_h}\int_{e} [v]^2 \mathrm ds. 
\end{equation}

%From \cite{Liu-2015},  In the rest of this paper, 
% we will mainly study superconvergence behaviour of DDG scheme for two-dimensional nonlinear convection-diffusion equations.

% inverse inequation?

\section{Construction of a special projection of the exact solution.}
 In this section, we will discuss the construction of a special projection $u_I$ of the exact solution, which is superconvergent towards the numerical solution 
 $u_h$ in the $L^2$ norm.  To this end, we begin with the error equation of the DDG scheme, which is the basis of the construction of $u_I$.

 \subsection{Error equations} 
   In the rest of this paper, we will use the following  notation 
 \begin{equation}\label{denote}
	\xi = u_h-u_I,\quad \eta =u-u_I,\ \quad e_h=u-u_h. 
\end{equation}
 Note that the exact solution u of \eqref{2dCD} also satisfies 
\eqref{DDGglo}.  Then  for all $v_h\in V_h^k$,
\begin{equation}\label{exactDDG}
	\left( \partial_tu,v_h\right) + A(u,v_h) = F(u,v_h).
\end{equation}
Subtracting \eqref{DDGglo} from \eqref{exactDDG}  yields the following error equation
\begin{equation}\label{errorequa}
	\left( \partial_t e_h,v_h\right) + A(e_h,v_h) = F(u,v_h)-F(u_h,v_h). 
\end{equation}
  Taking  $v_h=u_h-u_I$ in  \eqref{errorequa},  we have 
\begin{equation}\label{errorequa1}
	\left( \partial_t\xi,\xi\right) + A(\xi,\xi) = \left( \partial_t\eta,\xi\right) + A(\eta,\xi) +H, \quad H := -F(u,v_h)+F(u_h,v_h).
\end{equation}
 In light of  \eqref{coercivity},  we get 
\begin{equation*}
	\frac 1 2 \frac{d}{dt}\|\xi\|_0^2+\|\xi\|_E^2 \leq \left( \partial_t \eta,\xi\right) + A(\eta,\xi) +H.
\end{equation*}
   On  the other hand,  we have from the Taylor expansion that 
	\begin{equation}\label{Taylor}
			{\bf f}(u_h)-{\bf f}(u)=-{\bf f}'(u)e_h+\frac{{\bf f}''}{2} (e_h)^2,\ \ 
			{\bf f}(\{u_h\})-{\bf f}(u)=-{\bf f}'(u)(\{e_h\}) + \frac{\widetilde{ {\bf f}}''}{2} (\{e_h\}) ^2,
	\end{equation}
	where ${\bf f}''={\bf f}''(\theta_1 u+(1-\theta_1) u_h)$ and $\widetilde{{\bf f ''}}={\widetilde {\bf f}}''(\theta_2 u+(1-\theta_2) u_h)$  with $0<\theta_1,\theta_2<1$. 
  Consequently, 
	\begin{eqnarray*} 
			H &=& ( {\bf f}(u_h)- {\bf f}(u),\nabla \xi)+\sum_{e\in{\cal E}_h} \int_e [\xi]\left( \widehat{{\bf f} (u_h)}  - {{\bf f} (\{u_h\})} + {{\bf f} (\{u_h\})}-  {\bf f} (u)  \right)\cdot {\bf n} \mathrm ds \\ 
			   &=& ( -{\bf f}' e_h+\frac{{\bf f}''}{2}|e_h|^2,\nabla \xi)+\sum_{e\in{\cal E}_h} \int_e [\xi]\left( \widehat{{\bf f} (u_h)}  -{\bf f} (\{u_h\}) - {\bf f'}(u)\{e_h\} +\frac{\widetilde{ {\bf f}}''}{2} \{e_h\}^2\right)\cdot {\bf n} \mathrm ds \\ 
			   &=& H_1+H_2+H_3+H_4,
		%	&=&\displaystyle\sum_{\tau_{i,j}}  (f_1(u_h)-f_1(u), \xi_x)_{\tau_{i,j}} + \displaystyle\sum_{\tau_{i,j}} \int_{\tau_j^y}(\widehat{f_1(u_h)}-f_1(\{u_h\})+f_1(\{u_h\})-f_1(u))[\xi]|_{x_{i+ \frac 1 2}}\mathrm dy\\
		%	& + & \displaystyle\sum_{\tau_{i,j}}  (f_2(u_h)-f_2(u), \xi_y)_{\tau_{i,j}} + \displaystyle\sum_{\tau_{i,j}} \int_{\tau _i^x}(\widehat{f_2(u_h)}-f_1(\{u_h\})+f_1(\{u_h\})-f_2(u))[\xi]|_{y_{j+ \frac 1 2}}\mathrm dx.
	\end{eqnarray*}
	where
\begin{eqnarray*}
		&H_{1}& =\sum_{\tau\in{\cal T}_h} ({\bf f}'(u)\xi,\nabla \xi)_{\tau} +  \sum_{\tau\in{\cal T}_h} \int_{\partial\tau} {\bf f'}(u)\{\xi\} [\xi] \cdot {\bf n} \mathrm ds,\\
		&H_{2} &=-\sum_{\tau\in{\cal T}_h} ({\bf f}'(u)\eta,\nabla \xi)_{\tau} - \sum_{\tau\in{\cal T}_h} \int_{\partial\tau} {\bf f'}(u)\{\eta\} [\xi] \cdot  {\bf n}  \mathrm ds,\\
		&H_{3}& = \frac 1 2 \displaystyle\sum_{\tau\in{\cal T}_h} ({\bf f}''|e_h|^2,\nabla \xi)_{\tau} + \frac 1 2\sum_{\tau\in{\cal T}_h} \int_{\partial\tau} \widetilde{ {\bf f}}''\{e_h\}^2[\xi] \cdot{\bf n} \mathrm ds,\\
		&H_{4} &= \sum_{e\in{\cal E}_h} \int_e [\xi]\left( \widehat{{\bf f} (u_h)}  -{\bf f} (\{u_h\}) \right)\cdot {\bf n} \mathrm  ds 
\end{eqnarray*}
  We next estimate $H_i, i\le 4$, respectively. 
 As for  $H_{1}$, a simply integration by parts  yields 
\begin{eqnarray*}
			H_{1} & = &-\frac 1 2\displaystyle\sum_{\tau\in{\cal T}_h} (\nabla {\bf f}'(u),\xi^2)_{\tau} - \frac 1 2 \sum_{\tau\in{\cal T}_h} \int_{\partial\tau} {\bf f'}(u) [\xi]^2\cdot {\bf n} \mathrm ds +
			\sum_{\tau\in{\cal T}_h} \int_{\partial\tau} {\bf f'}(u)\{\xi\} [\xi] \cdot {\bf n} \mathrm ds\\
			& \le & -\frac 1 2\displaystyle\sum_{\tau\in{\cal T}_h} (\nabla {\bf f}'(u),\xi^2)_{\tau} \leq C\|\xi\|_0^2.
\end{eqnarray*}
	To estimate $H_{3}$, we use the Cauchy-Schwarz inequality to derive that 
	\begin{eqnarray*}
				H_{3}& \leq &   C \|e_h\|_{0,\infty} \|e_h\|_0\|\nabla \xi\|_0+
				\frac \gamma 4  \displaystyle\sum_{\tau\in{\cal T}_h} \int_{\partial \tau} \frac{\beta_0 }{h}[\xi]^2 ds + C \|e_h\|^4_{0,\infty}\\
			& \leq &  \frac \gamma 4 \|\xi\|^2_{E}+C\|e_h\|^4_{0,\infty}. 
\end{eqnarray*}
  As it was proved in  \cite{Liu-2015} that $ \|e_h\|_0\lesssim h^{k+1}\|u\|_{k+1}$, 
  we have $\|e_h\|_{0,\infty}\lesssim h^{k+\frac 12}\|u\|_{k+1}$, 
     and thus 
 \[
     H_{3} \le   \frac \gamma 4 \|\xi\|^2_{E}+ C h^{4k+2}\|u\|_{k+1}^2. 
 \]
   To estimate $H_4$, we define 
\begin{equation}\label{alpha}
 \alpha(\widehat{{\bf f}},\zeta):=[u]^{-1}(\widehat{{\bf f}}(u^-,u^+)-{\bf f}(\zeta)):=(\alpha_1,\alpha_2)
 \end{equation} 
  for any function $u$ and any point $\zeta$ between $u^-$ and $u^+$.  Then 
\begin{eqnarray*}
			H_4= \sum_{e\in{\cal E}_h} \int_e [\xi][u_h] \alpha(\widehat{{\bf f}},\{u_h\})\cdot {\bf n} \mathrm  ds 
			\le -
 			\sum_{e\in{\cal E}_h} \int_e [\xi][\eta] \alpha(\widehat{{\bf f}},\{u_h\})\cdot {\bf n} \mathrm  ds, 
\end{eqnarray*}
  where in the last step, we have used the identity  $[u_h]=[\xi]-[\eta]$. 	 Substituting the estimates of $H_i$ into the formula of $H$ yields 
		\begin{equation}\label{H00}
			 H \le  -\mathcal{F}(\eta,\xi) + \frac \gamma 4\|\xi\|_E^2 + Ch^{2(2k+1)}\|u\|_{k+1}^2 + C\|\xi\|_0^2, 
	\end{equation}
	where for any $\varphi, v$, 
\begin{align*}
	\mathcal{F}(\varphi,v)=& (\varphi{\bf f}'(u) ,
	\nabla  v) + \sum_{e\in{\cal E}_h} (\int_e [v]  \{\varphi\}{\bf f}'(u) \cdot {\bf n} \mathrm ds  + \int_e [v][\varphi]  \alpha(\widehat{{\bf f} (u_h)},\{u_h\}) \cdot  {\bf n} \mathrm ds).
\end{align*}
 Plugging \eqref{H00} into \eqref{errorequa1}, we have 
	\begin{equation}\label{eq:1}
				\frac 1 2 \frac{d}{dt}\|\xi\|_0^2 +\frac {3\gamma} {4}\|\xi\|_E^2  \leq   \left( \partial_t\eta,\xi\right) + A(\eta,\xi)+\mathcal{F}(\eta,\xi)+ Ch^{2(2k+1)}\|u\|_{k+1}^2 + C\|\xi\|_0^2.
		\end{equation}
Define 
\begin{equation}\label{a0}
	a(\varphi,v):=\left( \partial_t \varphi,v\right) + A(\varphi,v) -\mathcal{F}(\varphi,v),\quad \forall \varphi, v \in V_h^k. 
\end{equation}
 The error  inequality in \eqref{eq:1} indicates that 
 the error $\|u_h-u_I\|_0$ depends on  the error bound $a(u-u_I, \xi)$. In other words, to achieve  our superconvergence goal, the function 
 $u_I$ should be specially designed such that $a(u-u_I, v)$  is of high order for any function $v\in  V_h^k$.

  For any smooth function $v$ and any interval $\tau_i^x$,    we define  a function $P^{(x)}v \in \mathbb{P}^k(\tau^x_i)$ satisfying the following conditions:  
\begin{subequations}\label{Z}
	\begin{align}
		&\int_{\tau_i^x} (P^{(x)}v-v) \partial ^2_x w \mathrm dx=0, \quad \forall w \in \mathbb{P}^k(\tau_i^x),~ i \in \mathbb{Z}_{N_x}, \label{Za}\\
		&\partial_x\widehat{(P^{(x)}v)} |_{x_{i+\frac 1 2}}:= \beta_0(h_i^x)^{-1}[P^{(x)}v] + \{\partial_x(P^{(x)}v)\} + \beta_1h_i^x[\partial _x^2(P^{(x)}v)]|_{x_{i+\frac 1 2}} = \partial_x v|_{x_{i+\frac 1 2}}, \label{Zb}\\
		&\{P^{(x)}v\}|_{x_{i+\frac 1 2}}=\{v\}|_{x_{i+\frac 1 2}}, \quad i \in \mathbb{Z}_{N_x}. \label{Zc}	
	\end{align}
\end{subequations}
  Here $h_i^x=x_{i+\frac 12}-x_{i-\frac 12}$.  Note that when $k=1$,  $P^{(x)}v$ only needs to satisfy the conditions \eqref{Zb}-\eqref{Zc}. 
     Similarly, we can define $P^{(y)}v$ along the $y$-direction.   Define 
     $\Pi_h v \in V_h^k$ of $v$  by 
\begin{equation}\label{projection_glo}
	\Pi _h v = P^{(x)}\otimes P^{(y)} v. 
\end{equation}
It is shown in \cite{Liu-2015}
 that the global projection $\Pi_h$ is uniquely defined,
 and 
 \begin{equation}\label{optimalerror}
		\|v-P^{(x)} v\|_0+ \|v-P^{(y)} v\|_0+\|v-\Pi_h v\|_0  \lesssim h^{l+1}\|v\|_{l+1}, \quad 0\leq l \leq k. 
\end{equation}
Moreover,  a straightforward analysis using the approximation in \eqref{optimalerror} yields  
\begin{equation*}
	a(u-\Pi_h u,v_h) \lesssim h^{k+1}\|v_h\|_0,\quad \forall v_h \in V_h^k. 
\end{equation*}
  In other words, if we choose $u_I=\Pi_h u$, then optimal error estimates can be obtained, just as proved in  \cite{Liu-2015}. 
  Based on this optimal error estimates,  we construct a special correction function $\omega$ to correct the term 
  such that 
\begin{equation*}
	a(u-u_I, v_h)=a(u-\Pi_h u+\omega,v_h) \lesssim h^{k+1+l}\|v_h\|_0,\quad \forall v_h \in V_h^k 
\end{equation*}
  for some positive $l$.  In light of  \eqref{projection_glo},  we have 
\begin{equation}\label{sub_projection}
	u-\Pi_h u=E^xu+E^yu-E^xE^yu,
\end{equation}
where
\begin{equation}\label{subsub_projection}
\begin{split}
	E^xu=&~u-P^{(x)}u,\quad E^yu=u-P^{(y)}u,\\
	E^xE^yu&
	=P^{(y)}(P^{(x)}u-u)-(P^{(x)}u-u).
\end{split}
\end{equation}
Then
\begin{equation*}
	a(u-\Pi_h u,v)=a(E^xu,v) + a(E^yu,v) - a(E^xE^yu,v).
\end{equation*}
 In the following, we will construct fucntions $\omega$ and $\bar \omega$ to separately correct the error $a(E^xu,v)$ and $ a(E^yu,v)$ such that 
 $\omega=\omega+\bar \omega$,  and $a(E^xu+\omega,v), a(E^yu+\bar \omega,v)$ are both of high orders.

\subsection{Correction function for $a(E^xu,v)$}

We begin with some preliminaries. 
First,  
let  
\[
	D^{-1}_xv(s) = \int ^x_{x_{i-\frac 1 2}}v(s')\mathrm ds',~ x \in \tau_i^x. 
\]
It is obviously that 
\begin{equation}\label{D_norm}
	\|D^{-1}_xv\|_{0,\infty,\tau_i^x} \leq h_i^x\|v\|_{0,\infty,\tau_i^x} .
\end{equation}
Secondly, we denote by $Q_h^xv$  the Gauss-Lobatto   projection of any function of $v$  along the $x$-direction.  That is, $Q_h^xv|_{\tau _i^x} \in \mathbb{P} ^k(x) $ is a function satisfying 
\begin{subequations}\label{Q}
	\begin{align}
		&Q_h^xv(x^+_{i-\frac 1 2})=v(x^+_{i-\frac 1 2}),\quad Q_h^xv(x^-_{i+\frac 1 2})=v(x^-_{i+\frac 1 2}), \label{Qa}\\
		&\int_{x^+_{i-\frac 1 2}}^{x^-_{i+\frac 1 2}}(v-Q_h^xv) \phi \mathrm dx,\quad \forall \phi \in \mathbb{P} ^{k-2}(\tau _i^x),k \geq 2. \label{Qb}	
	\end{align}
\end{subequations}
 Theres hold the following  error estimates (see, e.g.,  \cite{Cao;zhang;zou:2kFVM,v-Iv}) 
\begin{equation}\label{Q_norm}
	\|v-Q_h^xv\|_{r,q} \lesssim h^{m-r+\frac{2}{q}} \|\partial_x^{m+1}v\|_{0}, ~ 0 \leq r \leq m\le k .
\end{equation}
 Similarly, we can define $Q^y_h$ along the $y$-direction, and the above approximation property also holds true for $Q_h^y$.

Thirdly,  recalling the definition of  $A(\cdot,\cdot)$ in \eqref{DDGsubterm} and using the integration by parts and the equation
\begin{equation}\label{sideofibp}
	[uv]=[u]\{v\} + \{u\}[v], 
	\end{equation} 
 we  get 
\begin{align}\label{Aetasimply}
		A(E^xu,\xi)=&-\displaystyle\sum_{\tau_{i,j}}( E^xu,  \xi _{xx})_{\tau_{i,j}}+\displaystyle\sum_{\tau_{i,j}}\int_{\tau_j^y} (\widehat{\partial_x (E^xu)}[\xi]-\{E^xu\}[\xi_x])|_{x_{i+\frac 1 2}}\mathrm  dy  \notag\\
		&+ \displaystyle\sum_{\tau_{i,j}} (\partial_y(E^xu), \xi_{y})_{\tau_{i,j}}  + \displaystyle\sum_{\tau _{i,j}}\int_{\tau _i^x} (\widehat{\partial_y(E^xu)}[\xi] + [E^xu]\{\xi_y\})|_{y_{j+\frac 1 2}} \mathrm dx.
  \end{align}
%   &\displaystyle\sum_{\tau_{i,j}} \iint _{\tau_{i,j}} \partial_x (E^xu) \xi_x\mathrm dx \mathrm dy+\displaystyle\sum_{\tau_{i,j}}\int_{\tau_j^y} \widehat{\partial_x (E^xu)}[\xi]|_{x_{i+\frac 1 2}}+[E^xu]\{\xi_x\}|_{x_{i+\frac 1 2}}\mathrm dy \notag\\
% 		&+\displaystyle\sum_{\tau_{i,j}} \iint _{\tau_{i,j}} \partial_y (E^xu) \xi_y\mathrm dx\mathrm dy+\displaystyle\sum_{\tau_{i,j}}\int_{\tau _i^x} \widehat{\partial_y (E^xu)}[\xi]|_{y_{j+\frac 1 2}}+[E^xu]\{\xi_y\}|_{y_{j+\frac 1 2}}\mathrm dx  \notag\\
% 		=
% \\
% 		&-\displaystyle\sum_{\tau_{i,j}} \iint _{\tau_{i,j}} E^xu  \xi _{yy}\mathrm dx\mathrm dy+\displaystyle\sum_{\tau_{i,j}}\int_{\tau _i^x} \widehat{\partial_y (E^xu)}[\xi]|_{y_{j+\frac 1 2}}-\{E^xu\}[\xi_y]|_{y_{j+\frac 1 2}}\mathrm dx
Using the definition of $P^{(x)} u$ in \eqref{Z} and the integration again yields  
\begin{equation*}
	\begin{split}
		A(E^xu,\xi) 
		= & -\displaystyle\sum_{\tau_{i,j}}  (\partial_{y}^2(E^xu), \xi)_{\tau_{i,j}} - \displaystyle\sum_{\tau _{i,j}}\int_{\tau _i^x} [\partial_y(E^xu)\xi]|_{y_{j+\frac 1 2}} \mathrm dx + \displaystyle\sum_{\tau _{i,j}}\int_{\tau _i^x}(\widehat{\partial_y(E^x u)})[\xi]|_{y_{j+\frac 1 2}}\mathrm dx\\
		= & -\displaystyle\sum_{\tau_{i,j}}  (\partial_y^2(E^xu), \xi)_{\tau_{i,j}},
	\end{split}
\end{equation*}
where in  the second step we have used the fact that $\partial_y^r E^xu, r\geq 0$ are continuous  about $y$. 

Similarly,  we recall the definition of $a(\cdot,\cdot)$ in \eqref{a0}
\begin{eqnarray}\nonumber
		a(E^xu,v)
		&= &\left( (\partial_t-\partial_y^2)E^x u,v \right)+(E^xu \nabla {\bf f}(u) ,\nabla v)+\sum_{e\in{\cal E}_h} (\int_e [v]  \{E^xu\}{\bf f}'(u) \cdot {\bf n} \mathrm ds  + \int_e [v][E^xu]  {\bf \alpha } \cdot  {\bf n} \mathrm ds)\\\label{aExusimplify}
	&=&\left( (\partial_t-\partial_y^2)E^x u,v \right) - (f'_1(u)E^xu,v_x) + ((f'_2(u)E^xu)_y,v )
		 - \displaystyle\sum_{\tau_{i,j}} \int_{\tau_j^y} \alpha_1[E^xu][v]|_{x_{i+\frac 1 2}} \mathrm dy. 
\end{eqnarray}
Here $\alpha, \alpha_1$ are given in \eqref{alpha}, and in the last step, we have used integration by parts and the identites $\{E^xu\}(x_{i+\frac 12},\cdot)=0, [E^xu](\cdot,y_{j+\frac 12})=0$.

Now we are ready to construct the correction function corresponding to the term $a(E^xu, v)$.  
 Let  ${\omega}_0 = E^xu $ and we   define  $\omega_l|_{\tau_i^x} \in\mathbb P_{k}(x),  1\le l\le k-1$ by 
 \begin{subequations}\label{w1}
	\begin{align}
		&\langle {\omega_l}, \partial ^2_x v\rangle_{\tau _i^x}  = \mathcal{B} ({\omega_{l-1}},v)_{\tau _i^x},  \quad \forall v \in \mathbb{P}^k(\tau _i^x) \backslash  \mathbb{P}^1(\tau _i^x),~ i \in \mathbb{Z}_{N_x}, \label{w1a}\\
		& \widehat{({\omega_l})_x}|_{x_{i+\frac 1 2}} :=\beta_0(h)^{-1}[{\omega_l}] + \{\partial_x({\omega_l})\} + \beta_1h[\partial _x^2({\omega_l})]|_{x_{i+\frac 1 2}}= \alpha_1[ {\omega_{l-1}}]|_{x_{i+\frac 1 2}}, \label{w1b}\\
		&\{{\omega_l}\}|_{x_{i+\frac 1 2}}=0,\quad i \in \mathbb{Z}_{N_x},\label{w1c}
	\end{align}
\end{subequations}
where  $\alpha_1$ is defined in \eqref{alpha},  $\langle w,v\rangle_{\tau_i^x} =\int_{\tau_i^x}( uv)(x,\cdot) dx$,  and
\begin{equation*}
	\begin{split}
			\mathcal{B} (\omega ,v)_{\tau _i^x}=\langle (\partial_t-\partial_y^2)\omega ,v \rangle_{\tau _i^x} -\langle f'_1(u)\omega, \partial_xv\rangle_{\tau _i^x} + \langle (f'_2(u)\omega)_y,v\rangle_{\tau _i^x}.
	\end{split}
\end{equation*}

The properties of the proposed correction function $\omega_l,~1 \leq l \leq k-1$ in the following theorem are essential to the proof of superconvergence; see Lemma 3.1 below.

\begin{lemma}[]\label{correction_1 error}
	The function ${\omega_l},~ 1 \leq l \leq k-1$ defined in \eqref{w1} is uniquely determined. Moreover, if $u \in H^{k+l+2}(\Omega)$  and  ${\bf f}(u)\in C^2(\Omega)$, then 
	\begin{equation}\label{w1error}
		\|\partial_t^r  {\omega_l}\| \lesssim h^{k+l+1}\|u\|_{k+l+2r}, \quad \|\partial_y^r {\omega_l}\| \lesssim h^{k+l+1}\|u\|_{k+l+r},~~\forall r\ge 0.
	\end{equation}
\end{lemma}
  \begin{proof}   Since $\omega_l|_{\tau_i^x}\in\mathbb P_{k}$, we suppose for any $y\in [0,2\pi]$ that 
	\begin{equation*}
		\omega_l(x,y):=\displaystyle\sum_{m=0}^k c^l_{i,m}(y,t)L_{i,m}(x),\ \ \forall x\in\tau_i^x,
	\end{equation*}
where $L_{i,m}(x)$ is the Legendre polynomial of degree $m$ in interval $\tau _i^x$.
For any $v \in \mathbb{P} ^k \backslash \mathbb{P} ^1 (\tau_i^x)$,  noticing that $\partial_x^2 v \in \mathbb{P} ^{k-2}(\tau_i^x)$, then we choose $\partial_x^2 v = L_{i,m},~m \in \mathbb{Z}_{k-2} $ in \eqref{w1a} to get
\begin{equation}\label{cijm}
			c^l_{i,m} = \frac{2m+1}{2 h_i^x} \left(\langle (\partial_t-\partial_y^2)\omega_{l-1}+(f'_2(u)\omega_{l-1})_y ,D_x^{-1}D_x^{-1}L_{i,m} \rangle_{\tau _i^x} -\langle f'_1(u)\omega_{l-1}, D_x^{-1}L_{i,m}\rangle_{\tau _i^x}\right).
\end{equation}
  It remains to determine $c^l_{i,k-1}, c^l_{i,k}$. In light of the conditions  in \eqref{w1b}-\eqref{w1c},  we get 
\begin{equation*}
	\begin{split}
		&\displaystyle\sum_{m=k-1}^k(L_{m}(1)c^l_{i,m}+L_{m}(-1)c^l_{i+1,m})=b_1^{i},\ \ b_1^{i}= -\displaystyle\sum_{m=0}^{k-2}(L_{m}(1)c^l_{i,m}+L_{m}(-1)c^l_{i+1,m}),\\
		&\displaystyle\sum_{m=k-1}^k(g_0(m)c^l_{i,m}+g_1(m)c^l_{i+1,m})=b_2^{i},\  b_2^{i}= h\alpha_1[\widetilde{\omega_1^{l-1}}]|_{x_{i+\frac 1 2}}-\displaystyle\sum_{m=0}^{k-2}(g_0(m)c^l_{i,m}+g_1(m)c^l_{i+1,m}),
	\end{split}
\end{equation*}
where
\[
   g_0(m) = -\beta_0 L_{m}(1) + L'_{m}(1) - 4\beta_1 L''_{m}(1),\  g_1(m) = \beta_0 L_{m}(-1) + L'_{m}(-1) + 4\beta_1 L''_{m}(-1).
\] 
Set an $N_x \times N_x$ block circulant matrix called M, with the first row $[A~B~O \cdots O]$ and the last row $[B~O \cdots O~A]$, where O is an $2\times2$ zero matrix, and 
\begin{align*}
	A=\left(
	\begin{matrix}
	L_{k-1}(1) &L_{k}(1)\\
	g_0(k-1) &g_0(k)\\
	\end{matrix}\right),
	B=\left(
	\begin{matrix}
	L_{k-1}(-1) &L_{k}(-1)\\
	g_1(k-1) &g_1(k)\\
	\end{matrix}\right).
\end{align*}
Then the  coefficients ${\bf c}=(c_1,\cdots,c_{N_x})^T$ with $c_i=(c_{i,k-1}^l, c_{i,k}^l)$ 
  satisfy the linear system 
\begin{equation}\label{mc}
		M{\bf c}={\bf b}, \  {\bf b}=(b_1,\cdots,b_{N_x})^T,\ b_i=(b_1^i,b_2^i). 
\end{equation}
It is proved in \cite{Liu-2015} that $M$ is non-singular when $\beta_0\geq \Gamma (\beta_1)$. Therefore, $c^l_{i,k-1},c_{i,k}^l$ are uniquely determined and thus 
 the uniqueness of  $\omega_l$ follows.

  We next estimate the coefficients $c^l_{i,m}$. 
 By a direct calculation  from \eqref{cijm} and  \eqref{D_norm},     we get 
  \[
    |c^l_{i,m}| \lesssim h ( \| (\partial_t-\partial_y^2)\omega_{l-1}\|_{0,1,\tau_i^x}+ \| \partial_y\omega_{l-1}\|_{0,1,\tau_i^x})+ \|\omega_{l-1} \|_{0,1,\tau_i^x},\ m\le k-2. 
 \]
   Consequently,  for all $\tau_{i,j}=\tau_i^x\times\tau_j^y$, 
 \begin{equation}\label{eq:2}
    \|c^l_{i,m}\|_{0,\tau_j^y}\lesssim  h^{\frac 32}  ( \| (\partial_t-\partial_y^2)\omega_{l-1}\|_{0,\tau_{i,j}}+ \| \partial_y\omega_{l-1}\|_{0,\tau_{i,j}})+ h^{\frac 12} \|\omega_{l-1} \|_{0,\tau_{i,j}},\ \ m\le k-2. 
 \end{equation}
 On the other hand,  since M is a circulant matrix,  we have from  \eqref{mc} that 
\[
    \|{\bf c}\|^2_0:=\sum_{i=1}^{N_x} \left(|c_{i,k-1}^l|^2+ |c_{i,k}^l|^2\right)\lesssim  \|M^{-1}\|^2_0\|{\bf b}\|^2_0\lesssim \sum_{i=1}^{N_x} \left(|b_{1}^i| ^2+ |b_{2}^i|^2\right), 
\]
 which yields 
\[
   \sum_{i=1}^{N_x} \left(\|c_{i,k-1}^l\|_{0,\tau_j^y}^2+ \|c_{i,k}^l\|_{0,\tau_j^y}^2\right)\lesssim \sum_{i=1}^{N_x} \left(\|b_{1}^i\|_{0,\tau_j^y}^2+ \|b_{2}^i\|_{0,\tau_j^y}^2\right). 
\] 
  Note that 
 \begin{eqnarray*}
	 \sum_{i=1}^{N_x}(\|b_1^{i}\|_{0,\tau_j^y}^2+\|b_2^{i}\|_{0,\tau_j^y}^2) & \lesssim & \sum_{i=1}^{N_x} \sum_{m=0}^{k-2} \|c^l_{i,m}\|^2_{0,\tau_j^y}
		+ h^2\langle[{\omega_{l-1}}]|_{x_{i+\frac 1 2}}, [{\omega_{l-1}}]|_{x_{i+\frac 1 2}}\rangle_{0,\tau_j^y}\\
		&\lesssim & \sum_{i=1}^{N_x} \sum_{m=0}^{k-2} \|c^l_{i,m}\|^2_{0,\tau_j^y}
		+ h\|\omega_{l-1}\|_{0,\tau_{i,j}}^2. 
\end{eqnarray*}
 Then 
 \[
   \sum_{i=1}^{N_x} \left(\|c_{i,k-1}^l\|_{0,\tau_j^y}^2+ \|c_{i,k}^l\|_{0,\tau_j^y}^2\right)\lesssim  \sum_{i=1}^{N_x} h\|\omega_{l-1}\|_{0,\tau_{i,j}}^2, 
\] 
 which yields, together with \eqref{eq:2} that 
  \[
  \sum_{i=1}^{N_x} \sum_{m=0}^{k} \|c^l_{i,m}\|^2_{0,\tau_j^y} \lesssim 
		h\sum_{i=1}^{N_x}  (\|\omega_{l-1}\|_{0,\tau_{i,j}}^2+ h^2  \| (\partial_t-\partial_y^2)\omega_{l-1}\|^2_{0,\tau_{i,j}}+ h^2 \| \partial_y\omega_{l-1}\|^2_{0,\tau_{i,j}})
\] 
  Consequently, 
\[
   \|\omega_l\|_0^2\lesssim \sum_{i=1}^{N_x}\sum_{j=1}^{N_y} h_i^x\sum_{m=0}^{k} \|c^l_{i,m}\|^2_{0,\tau_j^y} \lesssim h^2 (\|\omega_{l-1}\|_{0}^2+ h^2 \| (\partial_t-\partial_y^2)\omega_{l-1}\|^2_{0}+ h^2 \| \partial_y\omega_{l-1}\|^2_{0}). 
\]
Taking time derivative or  $y$-direction derivative on both side of  the above equation and following the same argument, we have for all $m, r\ge 0$
  \begin{equation*}
\|(\partial_t^r+\partial_y^m)\omega_l\|_0\lesssim	h (\|(\partial_t^r+\partial_y^m)\omega_{l-1}\|_{0}+ h^2\| (\partial_t^r+\partial_y^m)(\partial_t-\partial_y^2)\omega_{l-1}\|_{0}+h^2\|\partial_y(\partial_t^r+\partial_y^m)\omega_{l-1}\|_0).
\end{equation*} 
  % By recursion, there holds 
%\begin{eqnarray*}
%    \|\omega_l\|_0 &\lesssim & h (\|\omega_{l-1}\|_{0}+ h^2\| (\partial_t-\partial_y^2)\omega_{l-1}\|_{0}+h^2\|\partial_y\omega_{l-1}\|_0)\\
 %     &\lesssim & h^l (\|\omega_0\|_{0}+ h^{2l}\| (\partial_t-\partial_y^2)^l\omega_0\|_{0}+h^2\|\partial^l_y\omega_0\|_0). 
%\end{eqnarray*}
 By choosing $l = 1$ and using  the estimate of the projection $P^{(x)}$, we obtain
\begin{eqnarray*}
	&& \|\omega_1\|_0\lesssim h (\|E^xu\|_{0}+ h^2\| (\partial_t-\partial_y^2)E^xu\|_{0}+h^2\|\partial_yE^xu\|_0)\lesssim h^{k+2}\|u\|_{k+1},\\
	&& \|\partial_t\omega_1\|_0\lesssim h^{k+2}\|\partial_t u\|_{k+1}\lesssim h^{k+2}\|u\|_{k+3},\ \|\partial_y^r\omega_1\|_0\lesssim h^{k+2}\| u\|_{k+1+r}, \ r\le 2. 
\end{eqnarray*}
  Then  \eqref{w1error} follows from the method of recursion. 
\end{proof}

% Now we choose $\omega_1^p = \displaystyle\sum_{l=1}^p\omega_1 ^l =\displaystyle\sum_{l=1}^p  Q^y_h(\widetilde{\omega_1^l})$, $1\leq p\leq k-1$
% as the correction functions corresponding to the term $a(E^xu,v)$.
% Then we can get the following correction result.

\begin{theorem}[]\label{correction_1result}
 Suppose all the conditions of Lemma \ref{correction_1 error} hold. Let $u \in H^{k+p+2}(\Omega)$ be the solution of \eqref{2dCD}, $u_h$ be the numberical solution of \eqref{DDGloc}, and $\omega_l, l\le p\le k-1$   be defined in \eqref{w1}. Then  for all $v \in V_h^k$, 
\begin{equation}\label{correction_1eqresult}
	|a(E^xu + \sum_{l=1}^pQ_h^y\omega_l,v)|\leq Ch^{2(k+p+1)}\|u\|_{k+p+2}^2 + \frac \gamma 4 \|v\|_E^2+ \|v\|_0^2.
 \end{equation}
\end{theorem}
\begin{proof}
	  First, we observe that 
\begin{equation}\label{aExuw1}
		 a(E^x u + \sum_{l=1}^pQ_h^y\omega_l,v) = a(E^x u + \sum_{l=1}^p\omega_l,v)-\displaystyle\sum_{l=1}^p a(\omega_l-Q_h^y{\omega_l},v).
\end{equation}
  We next estimate  the two terms appeared in the right side of the above equation. 
 
 	\emph{Step 1:} 
Let 
\[
   b(w,v)=\displaystyle\sum_{\tau\in{\cal T}_h} \left( (\partial_t-\partial_y^2) {w} ,v \right)_{\tau} - \left( f'_1(u)w, v_x \right)_{\tau}+
			 \left( (f'_2(u)w)_y,v \right) _{\tau} -\displaystyle\sum_{\tau_{i,j}} \int_{\tau_j^y} \alpha_1 [{w}][v]|_{x_{i+\frac 1 2}} \mathrm dy.
\]
Using  the Cauchy-Schwarz inequality and the inverse inequality yields 
\begin{eqnarray*}
   |b(w,v)| & \lesssim & C (\|\partial_t w\|^2_0+\sum_{r=0}^2\|\partial_y^rw\|^2_0+\|v\|_0^2+ h\sum_{\tau_{i,j}} \int_{\tau_j^y}  [{w}]^2|_{x_{i+\frac 1 2}} \mathrm dy)+\frac{\gamma}{8} \|v\|_E^2\\
      & \lesssim & C (\|\partial_t w\|^2_0+\sum_{r=0}^2\|\partial_y^rw\|^2_0+\|v\|_0^2)+\frac{\gamma}{8} \|v\|_{E}^2. 
\end{eqnarray*}
On the other hand,   by \eqref{w1} and  \eqref{aExusimplify}, we derive
\[
			|a(E^x u +\displaystyle\sum_{l=1}^p {\omega_l},v) |=|b(\omega_p,v)|\lesssim C (\|\partial_t \omega_p\|^2_0+\sum_{r=0}^2\|\partial_y^r\omega_p\|^2_0+\|v\|_0^2)+\frac{\gamma}{8} \|v\|_{E}^2. 
\]
 In light of  \eqref{w1error}, we obtain
\[
	|a(E^x u +\displaystyle\sum_{l=1}^p {\omega_l},v) |\le C h^{2(k+p+1)}\|u\|^2_{k+p+2} +C \|v\|_0^2+\frac{\gamma}{8} \|v\|_{E}^2. 
\]
 
\emph{Step 2:}  Recalling the definition of $a(\cdot,\cdot)$, we have 
\begin{equation}\label{eq:7}
  a(\omega_l- Q_h^y\omega_l,v)
		=  ( \partial_t (\omega_l-Q_h^y{\omega_l} ),v ) + A(\omega_l-Q_h^y{\omega_l},v) + \mathcal{F}(\omega_l-Q_h^y{\omega_l},v).
\end{equation}
  In light of \eqref{Q_norm} and \eqref{w1error}, we easily get for all $m\le k+1$ that 
 \begin{eqnarray}\label{eq:4}
    | ( \partial_t (\omega_l-Q_h^y{\omega_l} ),v ) |\lesssim h^{m}\|\partial_t\partial_y^m\omega_l\|_0\|v\|_0\lesssim h^{k+l+m+1}\|u\|_{k+l+m+2}\|v\|_0. 
 \end{eqnarray} 
To estimate the term $A(\omega_l- Q_h^y{\omega_l},v)$, we first note that $Q_h^y$
is the Lobatto projection along $y$-direction, which yields  
\[
    \{Q_h^y{\omega_l} \}(x_{i+\frac 1 2},y)=\{\omega_l\}(x_{i+\frac 1 2},y) ,\ \  \widehat{ (\omega_l )_x}(x_{i+\frac 1 2},y)= \widehat{(Q_h^y\omega_l)_x} (x_{i+\frac 1 2},y).
 \]
 For convenience, we adopt  the following notation
 \[
 	\|v\|_{E_x}^2 = (v_x,v_x)+ \sum_{\tau_{i,j}\in{\mathcal{T} }_h}\int_{\tau_j^y} \frac{\beta_0 }{h} [v]^2|_{x_{i+\frac 1 2}} \mathrm dy,~~
	 \|v\|_{E_y}^2 = (v_y,v_y)+ \sum_{\tau_{i,j}\in{\mathcal{T} }_h}\int_{\tau_i^x} \frac{\beta_0 }{h} [v]^2|_{y_{j+\frac 1 2}} \mathrm dx.
\]
which implies that $\|v\|_E^2 = \|v\|_{E_x}^2+\|v\|_{E_y}^2$.

   Then,  we  use  \eqref{DDGsubterm} and  the integration by parts to derive 
\begin{eqnarray*}
		A(\omega_l-Q_h^y{\omega_l},v) 
		=-(\omega_l-Q_h^y{\omega_l},v_{xx})+ (\partial_y(\omega_l-Q_h^y{\omega_l}),v_{y})+\sum_{\tau_{i,j}}\int_{\tau _i^x}\widehat{ \partial_y(\omega_l-Q_h^y{\omega_l})}[v]|_{y_{j+\frac 1 2}} \mathrm dx.
		%\le C(h^{-4} \|\omega_l-Q_h^y{\omega_l}\|_0^2+\|v\|_0^2+h\sum_{\tau_{i,j}}\int_{\tau _i^x}\widehat{ \partial_y(\omega_l-Q_h^y{\omega_l})}^2|_{y_{j+\frac 1 2}} \mathrm dx )+\frac{\gamma}{4}\|v\|_{E}^2
\end{eqnarray*}
Then, a direct calculation from the Cauchy-Schwarz  and   the  inverse inequalities yields 
\begin{eqnarray}\nonumber
  | A(\omega_l-Q_h^y{\omega_l},v) |&\le &  C(h^{-4} \|\omega_l-Q_h^y{\omega_l}\|_0^2+\|v\|_0^2+h\sum\limits_{\tau_{i,j}}\int_{\tau _i^x}\widehat{ \partial_y(\omega_l-Q_h^y{\omega_l})}^2|_{y_{j+\frac 1 2}} \mathrm dx )+\frac{ \gamma\epsilon_l}{8}\|v\|_{E_y}^2\\\nonumber
  &\le& C(h^{-4} \|\omega_l-Q_h^y{\omega_l}\|_0^2+\|v\|_0^2+\sum\limits_{r=0}^2 h^{2r-1} \|\partial_y^r(\omega_l-Q_h^y{\omega_l})\|_{0,\infty}^2)+\frac{\gamma\epsilon_l}{8}\|v\|_{E_y}^2\\\label{eq:5}
  &\le& C h^{2(k+l+m-1)} \|u\|^2_{k+l+m} +\frac{\gamma\epsilon_l}{8}\|v\|_{E_y}^2+C\|v\|_0^2,\ \ m\le k. 
\end{eqnarray}
Here in the last step, we have used the estimates \eqref{w1error} and \eqref{Q_norm} with $Q_h^x$ replaced by $Q_h^y$. 
Similarly,  as for the last term $\mathcal{F}(\omega_l-Q_h^y{\omega_l},v)$, we use the properties of $Q_h^y$  to obtain 
\begin{eqnarray*}
			\mathcal{F}(\omega_l-Q_h^y{\omega_l},v) = & - \left(  f'_1(u)(\omega_l-Q_h^y{\omega_l}),v_x \right) + \left( (f'_2(u)(\omega_l-Q_h^y{\omega_l})_y,v\right)  \\
			&-\displaystyle\sum_{\tau_{i,j}} \int_{\tau_j^y} \alpha_1 
			[\omega_l-Q_h^y{\omega_l}][v]|_{x_{i+\frac 1 2}} \mathrm dy.
\end{eqnarray*}
  Again we use \eqref{w1error},  the Cauchy-Schwarz and the inverse inequalites  and then derive 
 \begin{eqnarray}\nonumber
    |\mathcal{F}(\omega_l-Q_h^y{\omega_l},v) | & \lesssim & \|\omega_l-Q_h^y{\omega_l}\|_0 \|v_x\|_{E_x}+  \|\partial_y(\omega_l-Q_h^y{\omega_l})\|_0 \|v\|_0\\\label{eq:6}
    &\le& C h^{2(k+l+m-1)} \|u\|^2_{k+l+m} +\frac{\gamma\epsilon_l}{8}\|v\|_{E_x}^2+C\|v\|_0^2,\ \ m\le k. 
 \end{eqnarray}
 Substituting \eqref{eq:4}-\eqref{eq:6}  into \eqref{eq:7} and taking $l+m=p+2$ yields 
\[
    | a(\omega_l- Q_h^y\omega_l,v) |\le C h^{2(k+p+1)} \|u\|^2_{k+p+2} +\frac{\gamma\epsilon_l}{8}\|v\|_{E}^2+C\|v\|_0^2.
\] 
 Choosing $\displaystyle{\sum_{l=1}^p}\epsilon_l=1$, there holds
\[
     \sum_{l=1}^p| a(\omega_l- Q_h^y\omega_l,v) | \le  C h^{2(k+p+1)} \|u\|^2_{k+p+2} +\frac{\gamma}{8}\|v\|_{E}^2+C\|v\|_0^2. 
 \]  
Then \eqref{correction_1eqresult} follows from \eqref{aExuw1}.

\end{proof}

\subsection{Correction function for $a(E^yu,v)$}

By the same argument as what we did for $a(E^xu,v)$ in \eqref{aExusimplify},  we can define the correction function $\bar w_l, l\le k-1$ for the term $a(E^xu,v)$. 
Denote $\bar w_0=E^yu$ and  $\bar w_l|_{\tau_j^y}\in\mathbb P_{k}(y)$
 is defined as 
\begin{subequations}\label{w2}
	\begin{align}
		&(\bar{\omega_{l}}, \partial ^2_y v)_{\tau_j^y}  = \mathcal{\bar B} (\bar{\omega}_{l-1},v)_{\tau_j^y},  \quad \forall v \in \mathbb{P}^k(\tau_j^y) \backslash  \mathbb{P}^1(\tau_j^y), j \in \mathbb{Z}_{N_y}, \label{w2a}\\
		&\widehat{(\bar{\omega_l})_y}|_{y_{j+\frac 1 2}}:=\beta_0(h)^{-1}[{\bar\omega_l}] + \{\partial_y\bar{\omega_l}\} + \beta_1h[\partial _y^2\bar{\omega_l}]|_{y_{j+\frac 1 2}} = \alpha_2[\bar{\omega_{l-1}}]|_{y_{j+\frac 1 2}}, \label{w2b}\\
		&\{\bar{\omega_l}\}|_{y_{j+\frac 1 2}}=0,\label{w2c}
	\end{align}
\end{subequations}
where  $\alpha_2$ is given in \eqref{alpha} and
\begin{equation*}
			\mathcal{\bar B} (w,v)_{\tau_j^y}=\langle (\partial_t-\partial_x^2)w,v \rangle_{\tau_j^y} -\langle f'_2(u)w, \partial_yv\rangle_{\tau_j^y} + \langle \partial_x(f'_1(u)w),v\rangle_{\tau_j^y}.
\end{equation*}

\begin{theorem}[]\label{correction_2result}
	 Suppose all the conditions of Lemma \ref{correction_1 error} hold. Let $u \in H^{k+p+2}(\Omega)$ be the solution of \eqref{2dCD}, $u_h$ be the numberical solution of \eqref{DDGloc}, and $\bar \omega_l, l\le p\le k-1$   be defined in \eqref{w1}. Then  for all $l\le p\le k-1$, 
	\begin{equation}\label{w2error}
		\|\partial_t^r  {\omega_l}\| \lesssim h^{k+l+1}\|u\|_{k+l+2}, \quad \|\partial_y^r {\omega_l}\| \lesssim h^{k+l+1}\|u\|_{k+l+r},~~\forall r\ge 0.
	\end{equation}
	Moreover, there holds for all $v \in V_h^k$
	\begin{equation}\label{correction_2eqresult}
		| a(E^yu +\sum_{l=1}^p \bar\omega_l,v)|\leq Ch^{2(k+p+1)}\|u\|_{k+p+2}^2 + \frac \gamma 4 \|v\|_E^2+ \|v\|_0^2. 
	\end{equation}
\end{theorem}
   Here we omit the proof since it is similar to that for $a(E^xu + \displaystyle\sum_{l=1}^p\omega_l,v)$.

\subsection{Estimate}
Define the final correction function 
\begin{equation}\label{finalcorrection^1}
	\omega:=\omega^p =  \displaystyle\sum_{l=1}^p  (Q_h^y\omega_l+Q_h^x\bar \omega_l),
\end{equation}
where ${\omega_l}$ and $\bar{\omega_l}$ are defined in \eqref{w1} and \eqref{w2}, respectively.
% According to conclusions in the first two subsection, we have the following estimates.

\begin{theorem}\label{correction^1final-result}
	Assume that $u\in H^{k+p+3},~ 1\leq p\leq k-1$ is the solution of \eqref{2dCD} and $\Pi_h u$ is the projection of $u$ defined in \eqref{projection_glo}.  Suppose that ${\bf f}(u)\in C^2$. 
		Then 
	\begin{equation}\label{correction^1final-eqresult}
		\left\lvert  a(u-\Pi_hu +\omega^p,v) \right\rvert\leq Ch^{2(k+p+1)}\|u\|_{k+p+3}^2 + \frac {3\gamma} {4}\|v\|_E^2+ C\|v\|_0^2, \quad 1\leq p\leq k-1.
	\end{equation}
\end{theorem}

\begin{proof}
	As  a direct consequence of \eqref{sub_projection}, \eqref{correction_1eqresult} and \eqref{correction_2eqresult},  we have for all $v \in V_h^k$
	\begin{equation}\label{correction^1onestep}
		\left\lvert  a(u-\Pi_hu +\omega^p,v) \right\rvert\leq Ch^{2(k+p+1)}\|u\|_{k+p+2}^2 + \frac {\gamma} 2 \|v\|_E^2+ 2\|v\|_0^2 +| a(E^xE^yu,v)|.
	\end{equation}
To estimate $a(E^xE^yu,v)$, we use the properties of $E^x$ and $E^y$ and the integration by parts to derive 
 \[
			A(E^xE^yu, v) =0,\ \ 
			\mathcal{F}(E^xE^yu,v)=(f'_1(u)E^xE^yu,v_x)+ (f'_2(u)E^xE^yu,v_y), 
\]
	 and thus 
\begin{eqnarray*}
			|a(E^xE^yu,v)|  &= &\left\lvert  \left(  \partial_t E^xE^y u,v \right)  - \left( f'_1(u)E^xE^yu,v_x\right)-\left(f'_2(u)E^xE^yu,v_y\right)\right\rvert\\
			&\le  &  \frac \gamma 4 \|v\|_E^2 +  \|v\|^2_0  + C(\|E^xE^yu\|^2_{0}+\|\partial^2_t E^xE^yu\|_{0}).
\end{eqnarray*}
  	 Here  $C$ is a constant independent of $h$. Using the estimates of $E^x, E^y$, we get 
\[
		\|E^yE^xu\|_{0} \lesssim h^{m} \|\partial_y^m E^xu\|_{0}\lesssim h^{m+r} \|\partial_y^m \partial_x^r u\|_{0},\ \ 0\le m,r\le k+1. 
\]
Similarly, there holds 
\[
	\|E^yE^xu_t\|_{0} \lesssim  h^{m+r} \|\partial_y^m \partial_x^r u_t\|_{0}\lesssim h^{m+r}\|u\|_{m+r+2},\ \ 0\le m,r\le k+1. 
\]
  Consequently, 
\begin{equation}\label{aExEyuestimate}
		\left\lvert  a(E^xE^y u, v) \right\rvert  \leq  \frac \gamma 4 \|v\|_E^2 + C \|v\|^2_0 + Ch^{2(k+1+p)} \|u\|^2_{k+p+3}.
\end{equation}
Substituting \eqref{aExEyuestimate}  into \eqref{correction^1onestep} yields the desired result \eqref{correction^1onestep}.
\end{proof}

\begin{remark}
	The estimates \eqref{aExEyuestimate} indicates that $a(E^xE^yu,v)$ is of high order  and no correction function is need for  $a(E^xE^yu,v)$. 
	\end{remark}

\section{Superconvergence}

% In this section, we study and reveal the superconvergence phenomenon of the DDG method for two-dimensional nonlinear convection diffusion equation. In the numerical examples, we observe a $(k + 2)$-th order superconvergence of the numerical solution at Lobatto points, $(2k)$-th order at nodes, and a $(k + 1)$-th order of the derivative approximation at Gauss points. All of these observations  will be confirmed in the rest of this section. 
Define
\begin{equation}\label{ui}
   u_I:=u_I^p=\Pi_hu-\omega^p. 
\end{equation}
To study the superconvergence behavior of the DDG solution,  we begin with the analysis of the supercloseness between the projection function $u_I$ and  $u_h$.

\begin{theorem}\label{supercolsenessu^1_I-u_h}
	Let $u\in H^{k+p+3}$ with $1\le p\le k-1$ be the solution of \eqref{2dCD}, and $u_I$ the special projection of $u$ defined in \eqref{ui}. Assume that  $u_h$  is the solution of the DDG scheme \eqref{DDGloc} with the initial solution chosen such that 
\begin{equation}\label{initialcondition}
	\|u_h - u_I\|(0) \lesssim h^{k+p+1}\|u_0\|_{k+p+3}.
\end{equation}
		Then 
	\begin{equation}\label{erroreu^1_I-u_h}
		\|u_I^p-u_h\|_0(t) \lesssim  h^{k+p+1} \|u\|_{k+p+3},\quad t>0.
	\end{equation}
\end{theorem}
\begin{proof}
	 As a direct consequence of \eqref{eq:1}-\eqref{a0} and \eqref{correction^1final-eqresult}, 
	\begin{equation}\label{errorequaleft}
		  \frac 1 2 \frac{d}{dt}\|\xi\|_0^2\lesssim  C \|\xi\|^2_0 + Ch^{2(k+1+p)} \|u\|^2_{k+p+3}.
	\end{equation}
Then the desired result follows from the Gronwall inequality and \eqref{initialcondition}. 
\end{proof}

\subsection{Superconvergence of numerical fluxes at nodes and for the cell avareges} 
 
\begin{theorem} \label{at nodes}
	Let $u\in H^{k+p+3}$ with $0\le p\le k-1$ be the solution of \eqref{2dCD}, and $u_I$ the special projection of $u$ defined in \eqref{ui}. Assume that  $u_h$  is the solution of the DDG scheme \eqref{DDGloc} with the initial solution chosen such that  \eqref{initialcondition} holds with $p=k-1$. 
		Then 
  \begin{eqnarray}\label{en}
 e_n:= \left( \frac 1 {N_xN_y} \displaystyle\sum_{i=1}^{N_x} \displaystyle\sum_{j=1}^{N_y} (u-\{u_h\})^2(x_{i+\frac 1 2},y_{j+\frac 1 2},t)\right)^{\frac 1 2}
		   \lesssim h^{2k} \|u \|_{2k+2}. 
		   %,\\\label{ec}
% && e_{c} : =\left( \frac{1}{N_xN_y}\sum_{\tau\in {\cal T}_h}\Big(\frac{1}{|\tau|}\int_{\tau}(u-u_h)\Big)^2dxdy\right)^{\frac
%  12}\lesssim h^{2k}  \|u \|_{2k+2}. 
 \end{eqnarray} 
\end{theorem}

\begin{proof}
	 On the one hand,  by choosing $p = k-1$ in 
	 \eqref{erroreu^1_I-u_h}, we obtain
	\begin{equation}\label{superclose value}
   		\|u^k_I-u_h\|(t)_0\lesssim h^{2k} \|u \|_{2k+2}.
	\end{equation}

On the one hand, we directly obtain from \eqref{w1c} and \eqref{w2c}
\[
    \{Q_h^y\omega_l\}({x_{i+\frac 1 2},y}) = \{\omega_l\}({x_{i+\frac 1 2},y})= 0,\
    \{Q_h^x\bar \omega_l\}({x, y_{i+\frac 1 2}}) = \{\bar \omega_l\}({x,y_{j+\frac 1 2},y})=0. 
\]
   Consequently, 
 \[
    \{\omega^{p}\}(x_{i+\frac 12}, y_{j+\frac 12})=0, \ \ \forall p\le k-1, 
 \]
   and thus 
	\begin{equation*}
		\{u\}(x_{i+\frac 12}, y_{j+\frac 12})= \{\Pi_h u \} (x_{i+\frac 12}, y_{j+\frac 12})= \{u_I^k \} (x_{i+\frac 12}, y_{j+\frac 12}).
			\end{equation*}
Then a direct calculation from the inverse inequality yields that 
\begin{eqnarray*}
    e_n=  \left( \frac 1 {N_xN_y} \displaystyle\sum_{i=1}^{N_x} \displaystyle\sum_{j=1}^{N_y} (\{u_I^{k-1}\}-\{u_h\})^2(x_{i+\frac 1 2},y_{j+\frac 1 2},t)\right)^{\frac 1 2}
			 \lesssim  \|u_I^{k-1}-u_h \|_0.
\end{eqnarray*}
Then  \eqref{en} follows from \eqref{superclose value}.  
\end{proof}

\subsection{Superconvergence at Lobatto and Gauss points} 

% First, we introduce a special projection, the Gauss-Lobatto projection, which palys an essential role in this confirmation.
Denote by $L_{\mu }$ the Legendre polynomial of degree ${\mu }$, and $\{\varphi_{\mu }\}_{\mu = 0}^\infty $
the series of Lobatto polynomials, on the interval $[-1,1]$. That is, 
$\varphi_0=\frac{1-s}{2}, \varphi_1=\frac{1+s}{2},$ and 
\begin{equation*}
	\varphi_{{\mu }+1}=\int^s_{-1}L_{\mu }(s')\mathrm ds', \quad {\mu } \geq 1,s\in[-1,1].
\end{equation*}
%Therefore we obtain Lobatto polynomials for two-dimensional by taking tensor product of those for one dimensional, namely 
%\begin{equation*}
%	\varphi_{{\mu },{\nu }}=\varphi_{{\mu }}(\xi)\varphi_{{\nu }}(\eta), \quad \forall (\xi,\eta) \in [-1,1]\times [-1,1] .
% \end{equation*}
Let $L_{i,{\mu }}$ and $\varphi_{i,{\mu }}$ be the Legendre and Lobatto polynomials of degree $\mu$ on the interval $\tau^x_i$ respectively. 
That is,
\begin{equation*}
	L_{i,\mu} (x)= L_\mu(\xi),\quad  \varphi_{i,\mu}(x)=\varphi_{\mu}(\xi),\quad \xi = \frac {2x-x_{i+\frac 1 2}- x_{i-\frac 1 2}}{2}, \quad \mu \geq 0.
\end{equation*}
Zeros of $L_{i,{k }}$ and $\varphi_{i,{k+1 }}$ are called the Gauss points of degree $k$ and Lobatto points of degree $k+1$ on $\tau_i^x$,  respectively. 
By tensor product, we can obtain the $k^2$ Gauss points ${\cal G}_\tau$ and $(k+1)^2$ Lobatto points ${\cal L}_{\tau}$ on each interval $\tau\in{\cal T}_h$. 
We denote by  ${\cal G}$ and   ${\cal L}$ the set of Gauss and Lobatto points on the  whole domain, respectively.

In  each element $\tau_i^x$,  denote by $I^{(x)}v\in \mathbb P_{k}(x)$ the Gauss-Lobatto projection of  $v$ along the $x$-direction satisfying 
\begin{equation}
	I^{(x)} v |_{\tau _i^x}= \displaystyle \sum _{\mu=0}^k  v_{i,\mu} \phi_{i,\mu}(x), 
\end{equation}
  where 
 \[
     v_{i,0} =v(x_{i-\frac 12}),\ \ v_{i,1} =v(x_{i+\frac 12}),\ \  v_{i,\mu}=\frac{2\mu-1}{2}\langle v_x, L_{i,\mu-1}\rangle_{\tau_i^x},\ \ \mu\ge 2. 
 \]
   Similarly, we can define the Gauss-Lobatto projection $I^{(y)}$ of  $v$ along the $y$-direction. 
% \begin{table}[!h]
% 	\centering
% 	\begin{tabular}{c|c|c|c}
% 		\hline
% 		 a & $p=0$ &$p = 1$ &$p2\leq p \leq k$
% 		% \hline
% 		% 1 & 2 &3 &4
% 		% $q = 0$ & $v(x^+_{i-\frac 1 2},y^+_{j-\frac 1 2})$ & $v(x^-_{i+\frac 1 2},y^+_{j-\frac 1 2})$ & $\frac{2 p-3}{2} \int_{I_{i}} \partial_{x} v (x,y^+_{j-\frac 1 2})L_{i, p-1} \math rm d x$
% 		% \hline
% 		% $q = 1$ & $v(x^+_{i-\frac 1 2},y^-_{j+\frac 1 2})$ & $v(x^-_{i+\frac 1 2},y^-_{j+\frac 1 2})$ & $\frac{2 p-3}{2} \int_{I_{i}} \partial_{x} v (x,y^-_{j+\frac 1 2})L_{i, p-1} \math rm d x$
% 		% \hline 
% 		% $2\leq q \leq k$ & $\frac{2 q-3}{2} \int_{I_{j}} \partial_{x} v (x,y^+_{j-\frac 1 2})L_{i, p-1} \math rm d x$ &$\frac{2 p-3}{2} \int_{I_{i}} \partial_{x} v (x,y^-_{j+\frac 1 2})L_{i, p-1} \math rm d x$ & $\frac{2 p-3}{2} \int_{I_{i}} \partial_{x} v (x,y^-_{j+\frac 1 2})L_{i, p-1} \math rm d x$
% 		\hline	
% 	\end{tabular}
% 	\caption{}
% \end{table}
 Define 
\begin{equation}
	I_h v = I^{(x)} \otimes I^{(y)} v. 
\end{equation}

\begin{lemma}[]\label{supercloseness}
	For any function $v \in H^{k+2}$, if $\beta_1=\frac{1}{2k(k+1)}$, then 
	\begin{equation}\label{eq:9}
		\|\Pi_h v-I_h v\|_0 \lesssim h^{k+2}\|v\|_{k+2}.
	\end{equation}
\end{lemma}
\begin{proof} When $\beta=\frac{1}{2k(k+1)}$, there holds for all $s\le k$ that (see, e.g., \cite{super-cao-cd}) 
	\begin{equation*}\label{superclosenessPi&I_h1d}
			\|P^{(x)}v-I^{(x)}v\|_{0,\tau^x_{i}} \lesssim h^{s+2}\|\partial_x^{s+2}v\|_{ 0,\tau^x_{i}}, \ \ 
			\|P^{(y)}v-I^{(y)}v\|_{0,\tau^y_{j}} \lesssim h^{s+2}\|\partial_y^{s+2}v\|_{ 0,\tau^y_{j}}. 
	\end{equation*}
  Then 
\begin{eqnarray*}
	\|\Pi_h v-I_h v\|_0  &=& \|P^{(x)}(P^{(y)}v) - I^{(x)}(I^{(y)}v)\|_0\\
	&\leq& \|P^{(x)}(P^{(y)}v - I^{(y)}v)\|_0 + \|I^{(y)}(P^{(x)}v- I^{(x)}v)\|_0\\
	&\lesssim & h^{k+2}\|v\|_{k+2}.
\end{eqnarray*}
  Here in the last step, we have used the fact both the Gauss-Lobatto projection $I^{(y)}$ and  the  projection $P^{(x)} $  are bounded.   The proof is complete. 
  \end{proof}

Now we are ready to present the superconvergence of DDG solution at Gauss points and Lobatto points. 

\begin{theorem}\label{Lobatto and Gauss}
	Let $u \in H^{k+4}$ be the solution of \eqref{2dCD}, and $u_h$ be the solution of \eqref{DDGloc} with the initial value $u_h (\cdot,\cdot, 0)$ chosen such that \eqref{initialcondition} holds with $p= 1$.
	If $\beta_1=\frac{1}{2k(k+1)}$, then  for any fixed $t \in (0,T]$
\begin{eqnarray}\label{Lobatto-Gauss}
			e_{g} :&=& \left(\frac 1 {N_x N_yk^2 }\displaystyle \sum_{\tau\in{\cal T}_h}\sum_{z\in{\cal G}_\tau}\nabla (u-u_h)^2(z,t)\right)^{\frac 1 2} 
			 \lesssim  h^{k+1} \|u\|_{k+4},\\\label{el}
			 e_l :&= & \left(\frac 1 {N_x N_y(k+1)^2 }\displaystyle\sum_{\tau\in{\cal T}_h}\sum_{z\in{\cal L}_\tau}(u-u_h)^2(z,t)\right)^{\frac 1 2} \lesssim  
		   h^{k+2}  \|u\|_{k+4}.
\end{eqnarray}
\end{theorem}
\begin{proof}   First, we choose $p=1$ in \eqref{erroreu^1_I-u_h} and use the triangle inequality to derive 
\[
    \|P_hu-u_h\|_0\lesssim  \|u_I^1-u_h\|_0+\|\omega^1\|_0\lesssim h^{k+2}\|u\|_{k+4}, 
\]
 which yields, together with \eqref{supercloseness} that 
\begin{equation}\label{eq:8}
   \|I_hu-u_h\|_0\lesssim h^{k+2}\|u\|_{k+4}. 
\end{equation}
 Using the inverse inequality, we have 
\begin{equation}\label{eq:9}
    \|I_hu-u_h\|_1\lesssim h^{k+1}\|u\|_{k+4}. 
\end{equation}
  On the other hand,  there holds (see, e.g., \cite{chen}) that 
\[
   | (u-I_hu)(z_0)|\lesssim h^{k+2}|u|_{k+2,\infty},\ \  |\nabla  (u-I_hu)(z_1)|\lesssim h^{k+1}|u|_{k+2,\infty},\ \ \forall z_0\in{\cal L}, z_1\in{\cal G}. 
\]
  Then 
 \begin{eqnarray*}
     e_l &\lesssim & \left(\frac 1 {N_x N_y }\displaystyle \sum_{\tau\in{\cal T}_h}\sum_{z\in{\cal L}}(I_hu-u_h)^2(z,t)\right)^{\frac 1 2} + \max_{z\in{\cal L}}|(u-I_hu)(z,t)|\\
		&\lesssim &\left( h^{-2}\sum_{\tau\in{\cal T}_h} \|I_hu-u_h\|^2_{0,\infty,\tau}\right)^{\frac 12}+h^{k+2}\|u\|_{k+2,\infty}\lesssim \|I_hu-u_h\|_0+h^{k+2}\|u\|_{k+2,\infty}. 
 \end{eqnarray*}
   Here in the last step, we have used the inverse inequality. 
 Similarly, there holds 
 \begin{eqnarray*}
     e_g &\lesssim &  \left(\frac 1 {N_x N_y }\displaystyle \sum_{\tau\in{\cal T}_h}\sum_{z\in{\cal G}}\nabla(I_hu-u_h)^2(z,t)\right)^{\frac 1 2} + \max_{z\in{\cal G}}|\nabla(u-I_hu)(z,t)|\\
		&\lesssim & \|I_hu-u_h\|_1+h^{k+2}\|u\|_{k+2,\infty}. 
 \end{eqnarray*}
Then the desired result  \eqref{Lobatto-Gauss} follows from \eqref{eq:8} and \eqref{eq:9}. 
\end{proof}

\begin{remark}\label{remark for beta_1}
	Same as the one-dimensional case,  the choice of $\beta_1$ in \eqref{dif-flux} has influence on the superconvergence  phenomenon of the DDG solution $u_h$ at Gauss and Lobatto points. 
	However, it does not affect the superconvergence rate of $u_h$ at nodes.  Our numerical examples  will demonstrate this fact. 
\end{remark}

 To end with this section, we would like to demonstrate how to obtain the initial solution. 
  To ensure the validity of  \eqref{initialcondition},   a nature method for 
  the initial discretization is to choose
\begin{equation*}
	u_h(x,y,0) = u_I^p(x,y,0), \ \ p\le k-1. 
\end{equation*}
 Then we divide the initial discretization into the following steps:

 1. According to the deﬁnition of $\Pi _h u$, compute ${\omega_l},{\bar \omega_l},1\leq l\leq p$ by \eqref{w1} and \eqref{w2}.

 2.  Calculate $Q_h^y\omega_l, Q_h^x\bar\omega_l, 1\leq l\leq p$. 

 3. Compute $\omega^p = \displaystyle\sum_{l=1}^{p}(Q_h^y\omega_l+ Q_h^x\bar \omega_l)$.

 4. Figure out $u_h(x,y,0)= \Pi_h u - \omega^p $.

\section{Numerical tests}

In this section, we provide numerical examples to validate our theoretical results.
% Due to the symmetry of the selected examples, in order to save palce, we only show the errors of $e_l,e_n,e_{gx}$ whose definations lie in \eqref{Lobatto&Gauss}, \eqref{nodes}, and $e_{gl_x}$ defined as
%\begin{equation}
%	e_{gl_x} = (\frac 1 {N_x N_y k k_p}\displaystyle \sum_{i=1}^{N_x}\sum_{j=1}^{N_y}\sum_{m=1}^{k}\sum_{n=1}^{k_p}\partial_x(u-u_h)^2(g_{i,m},y_{j,p}))^{\frac 1 2},
%\end{equation}
%where $y_{j,k_p}$ is a arbitrary subdivision of $\tau_j^y$, namely $y_{j-\frac 1 2} = y_{j,1} = \cdots = y_{j,k_p} = y_{j+\frac 1 2 }$, to observe the superconvergence phenomenon of $u_h$.
% Actually numerical performance of $e_{gl_x}$  means we even have superconvergence on the gauss lines in two dimensions. In this section, we choose $k_p = 100$ and uniform subdivision to verify the result on gauss lines. 
In our experiments, 
we adopt the DDG method as spatial discretization and the classic fourth order Runge-Kutta method as our time discretization  method. 
The convection flux is chosen as   Godnov  numerical flux. 
 The CFL condition is taken as $ \tau = {\cal O}(h^2)$, where $\tau$  is the temporal step size. 
 We obtain our meshes by equally dividing $[0,2\pi]\times [0,2\pi]$ into $N \times N$ subintervals with $N=4,\ldots 32$. 

\begin{example}
	Consider the following equation with the periodic boundary condition:
\begin{equation*}
	\begin{cases}
		u_t+(\frac {u^2} {2})_x+(\frac {u^2} {2})_y = u_{xx}+u_{yy} + g(x,y,t), & (x,y,t) \in [0,2\pi] \times [0,2\pi] \times (0,1]\\
		u(x,y,0) = sin(x+y),&(x,y,t) \in [0,2\pi] \times [0,2\pi], 
	\end{cases}
 \end{equation*}
where the source term $g(x, t)$ is  chosen such that the exact solution is
 \begin{equation*}
	u = e^{(-2t)}sin(x+y).
 \end{equation*}
 
\end{example}

	% equation*} 
	% 		\begin{split}
	% 		&u_t+(\frac {u^2} {2})_x+(\frac {u^2} {2})_y = u_{xx}+u_{yy} + g(x,y,t), \quad (x,y,t) \in [0,2\pi] \times [0,2\pi] \times (0,1]\\
	% 		&u(x,y,0) = sin(x+y),
	% 	\end{split}
	% \end{equation*}

	\begin{table}[!h]
		\centering
		\small{\caption{\label{Tab: bugurs1} \emph{Errors and convergence rates for $f_1(u)=f_2(u)=\frac1 2 u^2$ and $\beta_1 = \frac {1}{2k(k+1)}$.}}}
		\begin{tabular}{cccccccccc}
		% $~ $  &       &\multicolumn{5}{c}{RSV flux}& \multicolumn{5}{c}{RSV solution}\\
			\hline
			$k$&$N\times N $ &$e_{l}$   &rate    &$e_{n}$  &rate   &$e_{g_x}$  &rate     &$\|e_h\|_0$   &rate \\
			\hline
			\multirow{4}{*}{$1$}
			% &$~128$   &       &      &      &       &      &     &     &     &     &\\
			&$4\times 4$ &6.3e-03	&—	&2.4e-03	&—	&2.6e-02	&—		&1.6e-01	&—\\
			&$8\times 8$ &1.1e-03	&2.5 	&4.6e-04	&2.4 	&7.3e-03	&1.8 	 	&4.4e-02	&1.8 \\
			&$16\times 16$ &2.5e-04	&2.2 	&1.2e-04	&2.0 	&1.9e-03	&2.0 		&1.1e-02	&2.0\\ 
			&$32\times 32$ &6.2e-05	&2.0 	&3.0e-05	&2.0 	&4.7e-04	&2.0 	 	&2.8e-03	&2.0\\ 
			\hline
			\multirow{4}{*}{$2$}
			&$4\times 4$ &4.7e-04	 &—	 &1.7e-04	 &—	 &2.4e-03	 &—	  &1.4e-02	 &—\\
			&$8\times 8$ &2.1e-05	 &4.5 	 &3.2e-06	 &5.7 	 &2.9e-04	 &3.0 	 	 &1.7e-03	 &3.1 \\
			&$16\times 16$ &1.2e-06	 &4.1 	 &1.4e-07	 &4.5 	 &3.6e-05	 &3.0 	 	 &2.1e-04	 &3.0 \\
			&$32\times 32$ &7.6e-08	 &4.0 	 &9.5e-09	 &3.9 	&4.4e-06	 &3.0 	 &2.6e-05	 &3.0 \\
			\hline
			\multirow{4}{*}{$3$}
			&$4\times 4$ &2.7e-05 &	—	 &1.7e-05	 &—	 &2.1e-04	 &—	  &1.2e-03	 &—\\
			&$8\times 8$ &1.0e-06	 &4.8 	 &3.1e-07	 &5.8 	 &1.4e-05	 &3.9 		 &7.5e-05	 &4.0 \\
			&$16\times 16$ &3.3e-08	 &4.9 	 &4.9e-09	 &6.0 	 &9.1e-07	 &4.0 		 &4.7e-06	 &4.0 \\
			&$32\times 32$ &1.1e-09	 &5.0 	 &7.7e-11	 &6.0 	 &5.8e-08	 &4.0 		 &3.0e-07	 &4.0 \\
			\hline
			\multirow{4}{*}{$4$}
			&$4\times 4$ &1.9e-06	 &—	 &2.3e-07	 &—	 &1.7e-05	 &—	 	 &9.2e-05	 &—\\
			&$8\times 8$ &2.1e-08	 &6.4 	 &7.8e-10	 &8.2 	 &4.9e-07	 &5.1 		 &2.9e-06	 &5.0 \\
			&$16\times 16$ &2.9e-10	 &6.2 	 &3.1e-12	 &8.0 	 &1.5e-08	 &5.0 		 &9.0e-08	 &5.0 \\
			&$32\times 32$ &4.4e-12	 &6.1 	 &1.5e-14	 &7.7 	 &4.7e-10	 &5.0 	 &2.8e-09	 &5.0 \\
			\hline
		\end{tabular}
	\end{table}
	\begin{table}[!h]
		\centering
		\small{\caption{\label{Tab: bugurs2} \emph{Errors and convergence rates for $f_1(u)=f_2(u)=\frac1 2 u^2$ and $\beta_1 \neq \frac {1}{2k(k+1)}$.}}}
		\begin{tabular}{cccccccccc}
		% $~ $  &       &\multicolumn{5}{c}{RSV flux}& \multicolumn{5}{c}{RSV solution}\\
			\hline
			$k$&$ N\times N $ &$e_{l}$   &rate    &$e_{n}$  &rate   &$e_{g_x}$  &rate     &$\|e_h\|_0$   &rate \\
			\hline
			\multirow{4}{*}{$1$}
			&$4\times 4$ &6.3e-03	&—	&2.4e-03	&—	&2.6e-02	&—		&1.6e-01	&—\\
			&$8\times 8$ &1.1e-03	&2.5 	&4.6e-04	&2.4 	&7.3e-03	&1.8 	 	&4.4e-02	&1.8 \\
			&$16\times 16$ &2.5e-04	&2.2 	&1.2e-04	&2.0 	&1.9e-03	&2.0 		&1.1e-02	&2.0\\ 
			&$32\times 32$ &6.2e-05	&2.0 	&3.0e-05	&2.0 	&4.7e-04	&2.0 	 	&2.8e-03	&2.0\\
			\hline
			\multirow{4}{*}{$2$}
			&$4\times 4$ &6.7e-04	 &—	 &1.7e-04	 &—	 &1.9e-03	 &—	 &1.3e-02	 &—\\
			&$8\times 8$ &9.9e-05	 &2.8 	 &3.0e-06	 &5.8 	 &2.5e-04	 &2.9 		 &1.5e-03	 &3.1 \\
			&$16\times 16$ &1.3e-05	 &2.9 	 &1.2e-07 &4.7 	 &4.6e-05	 &2.4 	  	 &1.8e-04	 &3.0 \\
			&$32\times 32$ &1.7e-06	 &3.0 	 &7.4e-09	 &4.0 	 &1.0e-05	 &2.2 		 &2.2e-05	 &3.0 \\
			\hline
			\multirow{4}{*}{$3$}
			&$4\times 4$ &8.4e-05 &	—	 &1.5e-05	 &—	 &2.9e-04	 &—	  &1.1e-03	 &—\\
			&$8\times 8$ &3.3e-06	 &4.7 	 &3.0e-07	 &5.6 	 &2.2e-05	 &3.7 	 	 &7.4e-05	 &3.9 \\
			&$16\times 16$ &1.1e-07	 &4.9 	 &4.8e-09 	&5.9 	 &1.5e-06	 &3.9 	   &4.7e-06	 &4.0 \\
			&$32\times 32$ &3.5e-09	 &5.0 	 &7.6e-11	 &6.0 	 &9.3e-08	 &4.0 	 	 &3.0e-07	 &4.0 \\
			\hline
			\multirow{4}{*}{$4$}
			&$4\times 4$ &1.5e-05	 &—		 &3.1e-07	 &—	 &5.7e-05	 &—  &1.2e-04 	 &—\\
			&$8\times 8$ &5.2e-07	 &4.8 		&1.3e-09	 &7.9 	 &3.6e-06	 &4.0 		 &4.1e-06	 &4.9 \\
			&$16\times 16$ &1.7e-08	 &5.0 	&5.2e-12	 &7.9 	 &2.2e-07	 &4.0 	 	 &1.3e-07	 &5.0 \\
			&$32\times 32$ &5.3e-10	 &5.0 	&2.3e-14	 &7.8 	 &1.4e-08	 &4.0 		 &4.1e-09	 &5.0 \\
			\hline
			
		\end{tabular}
\end{table}

	% We take different  values of $(\beta_0 , \beta_1 ) $ to test the  influence of the choice of parameters in the DDG numerical  fluxes  \eqref{dif-flux} on the superconvergence rate.
 Listed  in Table \ref{Tab: bugurs1} are  various errors and the corresponding convergence rates for $k = 1,2,3,4$, and $T =1$, with the parameter$(\beta_0,\beta_1)$ taken as $(12,\displaystyle\frac 1 4)$$(12,\displaystyle\frac 1 {12})$$(12,\displaystyle\frac 1 {24})$$(12,\displaystyle\frac 1 {40})$ for $k=1,2,3,4$, respectively.  Note that $\beta_1 = \displaystyle\frac 1 {2k(k+1)}$ in this case. 
  From Talbe \ref{Tab: bugurs1},  we observe  a optimal convergence rate of $(k+1)$-th order for the $L^2$ error $\|e_h\|_0$, 
 a superconvergence rates of  $2k$ for the error $e_n$ at nodes, 
 $(k+2)$-th order for the error $e_l$ at Lobatto points and  $(k+1)$-th order for the error  $e_{g}$ at Gauss points. 
 These results are consistent with the  theoretical results given in  Theorems \ref{at nodes} and \ref{Lobatto and Gauss}, which indicates that 
 superconvergence   of the function value approximation at Lobatto points and nodes,  and superconvergence of the derivative approximation at Gauss points 
 exist. 
 
To test the infulence of $\beta_1$ on the superconvergence rate, we also test the case in which $\beta_1 \neq \displaystyle\frac {1}{2k(k+1)} $. Table \ref{Tab: bugurs2} demonstrates the error and corresponding convergence rate for $k=1,2,3,4$, with $(\beta_0,\beta_1) = (12,\displaystyle\frac 1 {40}),(12,\displaystyle\frac 1 4),(12,\displaystyle\frac 1 {12}),(12,\displaystyle\frac 1 {24}),$ respectively. 
As indicated by Table \ref{Tab: bugurs2},  the convergence rates of $e_l$ and $e_g$ in case  for $k=2,4$  are separately $k+1$ and $k$, which indicates that the 
  superconvergence  phenomena at Gauss and Lobatto  points disappear. 
While we still observe a supconvergence rate of $2k$-th order for the error $e_n$.
 In other words, the choice of $\beta_1$ affects the superconvergence of $u_h$ at Gauss and Lobatto points
while the superconvergence properties of $u_h$ at nodes are independent of the value of $\beta_1$ in the DDG schemes.

\begin{example}
	Consider the following equation with the periodic boundary condition:
\begin{equation*}
	\begin{cases}
		u_t+(sin(u))_x+(sin(u))_y = u_{xx}+u_{yy} + g(x,y,t), & (x,y,t) \in [0,2\pi] \times [0,2\pi] \times (0,1]\\
		u(x,y,0) = sin(x+y),&(x,y,t) \in [0,2\pi] \times [0,2\pi].
	\end{cases}
 \end{equation*}
We choose a suitable  $g(x, y, t)$  such that the exact solution is
 \begin{equation*}
	u = e^{(-2t)}sin(x+y).
 \end{equation*}

\end{example}

\begin{table}[!h]
	\centering
	\small{\caption{\label{Tab: sinu} \emph{Errors and convergence rates for $f_1(u)=f_2(u)=sin(u)$.}}}
	\begin{tabular}{cccccccccccc}
% $~ $  &       &\multicolumn{5}{c}{RSV flux}& \multicolumn{5}{c}{RSV solution}\\
		\hline
		$k$&$m\times n $ &$e_{l}$   &rate    &$e_{n}$  &rate   &$e_{g_x}$  &rate     &$L^2$   &rate \\
	\hline
		\multirow{4}{*}{$1$}
		% &$~128$   &       &      &      &       &      &     &     &     &     &\\
		&$4\times 4$ &8.2e-02	&—	&3.2e-02	&—	&3.1e-02	&—		&2.0e-01	&—\\
		&$8\times 8$ &1.8e-02	&2.2 	&7.9e-03	&2.0 	&8.9e-03	&1.8 		&5.4e-02	&1.9 \\
		&$16\times 16$ &4.3e-03	&2.0 &	2.0e-03	&2.0 	&2.3e-03	&2.0 		&1.4e-02	&2.0\\ 
		&$32\times 32$ &1.1e-03	&2.0 	&5.2e-04	&2.0 	&5.8e-04	&2.0 		&3.5e-03	&2.0 \\
		\hline
		\multirow{4}{*}{$2$}
		% &$~128$   &       &      &      &       &      &     &     &      &     & \\
		&$4\times 4$   &1.5e-03   &—  &4.9e-04  &—   &3.6e-03  &—  &2.0e-02 &—  \\
		&$8\times 8$   &1.0e-04   &3.9 &4.9e-05  &3.3  &4.6e-04  &3.0  &2.4e-03 &3.0  \\
		&$16\times 16$   &6.4e-06   &4.0 &3.4e-06  &3.8  &5.7e-05  &3.0  &3.0e-04 &3.0  \\
		&$32\times 32$   &4.0e-07   &4.0 &2.3e-07  &3.9  &7.2e-06  &3.0  &3.7e-05 &3.0  \\
		\hline
		\multirow{4}{*}{$3$}

		&$4\times 4$ &7.5e-05	&—	&2.6e-05	&—	&3.4e-04	&—		&1.7e-03	&— \\
		&$8\times 8$ &2.7e-06	&4.8 	&5.2e-07	&5.6 	&2.4e-05	&3.9 		&1.1e-04	&4.0 \\
		&$16\times 16$ &8.8e-08	&4.9 	&9.0e-09	&5.9 	&1.5e-06	&4.0 		&6.7e-06	&4.0 \\
		&$32\times 32$ &2.8e-09	&5.0 	&1.5e-10	&6.0 	&9.5e-08	&4.0 		&4.2e-07	&4.0 \\

		\hline
		\multirow{4}{*}{$4$}

		&$4\times 4$ &4.6e-06	&—	&5.4e-07	&—	&2.7e-05	&—		&1.3e-04	&—\\
		&$8\times 8$ &6.1e-08	&6.2 	&2.6e-09	&7.7 	&7.9e-07	&5.1 	 	&4.1e-06	&5.0 \\
		&$16\times 16$ &8.9e-10	&6.1 	&1.1e-11	&7.9 	&2.4e-08	&5.0 		&1.3e-07	&5.0\\ 
		&$32\times 32$ &1.4e-11	&6.0 	&6.2e-14	&7.4 	&7.4e-10	&5.0 		&4.0e-09	&5.0 \\

		\hline
	\end{tabular}

\end{table}
% In Example 1, we have tested the effect of $\beta_1$ on the phenomenon of superconvergence. Therefore, we only show the numerical experiments with $\beta_1 = \displaystyle\frac{1}{2k(k+1)}$.
We list in Table \ref{Tab: sinu} the approximation errors and corresponding convergence rates calculated by the DDG scheme for $k = 1,2,3,4$, 
 with the parameter chosen as $(\beta_0,\beta_1) = (12,\frac {1}{4}),(12,\frac {1}{12}),(12,\frac {1}{24}),(12,\frac {1}{40})$,  respectively. 
 As the same with Example 1, we observe the function value error is supercovergent at Lobatto points and the derivative error is superconvergent at Gauss points, with an order of $k+2$ and $k+1$, respectively. 
 Moreover, we see the error of numerical fluxes at mesh nodes are superconvergent with an order of $2k$. All these numerical results are consistent with the theoretical findings in  Theorems \ref{at nodes} and \ref{Lobatto and Gauss}. In other words, the error bounds established in   \eqref{en}  and \eqref{Lobatto-Gauss}-\eqref{el} are sharp.

\section{Conclusion}
In this paper, we investigate superconvergence properties of the DDG method for two-dimensional nonlinear convection-diffusion equations using tensor product meshes and tensor product polynomials of degree $k$. We prove that, with suitable initial  discretizations, the error between the DDG solution and the exact solution converges with an order of $(k+2)$ at  Lobatto points, with $(2k)$ at nodes,  and the convergence rate of derivative approximation  
at Gauss points can  achieve $(k+1)$.  Numerical experiments demonstrate that all the established error bounds  are optimal.
% The sharpness of the theoretical results is veriﬁed by numerical experiments.

% \section*{Appendix 1}

% \bibliography{yinyong}

\end{document}